%% file: sp2decomposn.tex
\documentclass[12pt]{article}
\usepackage{e-jc}

\input{preamble}

\makeatletter
	\renewcommand{\ps@plain}{%
	\renewcommand{\@oddfoot}{~%\footsc the electronic journal of combinatorics {\footbf\volno} (\volyear), \#\papno%
			~\hfil\footrm\thepage}}
\makeatother

\title{On restricted unitary Cayley graphs and symplectic transformations modulo $n$}

\author{%
  Niel de Beaudrap\footnote{%
	niel.debeaudrap@gmail.com}\\[0.5ex]
	\normalsize	Quantum Information Theory Group \\[-0.5ex]
	\normalsize Institut f\"ur Physik und Astronomie, Universit\"at Potsdam
}

\date{%\dateline{2010}{2010}\\
}

\begin{document}

\maketitle

\begin{abstract}
  We present some observations on a restricted variant of unitary Cayley graphs modulo $n$, and implications for a decomposition of elements of symplectic operators over the integers modulo $n$.
  We define \emph{quadratic unitary Cayley graphs} $G_n$, whose vertex set is the ring $\Z_n$, and where residues $a,b$ modulo $n$ are adjacent if and only if their difference is a quadratic residue.
  By bounding the diameter of such graphs, we show an upper bound on the number of elementary operations (symplectic scalar multiplications, symplectic row swaps, and row additions or subtractions) required to decompose a symplectic matrix over $\Z_n$.
  We also characterize the conditions on $n$ for $G_n$ to be a perfect graph.
\end{abstract}
\vspace{-0.5em}

%  ======================================================================
\section{Introduction}

For an integer $n \ge 1$, we denote the ring of integers modulo $n$ by $\Z_n$, and the group of multiplicative units modulo $n$ by $\Z_n\units$.
A well-studied family of graphs are the \emph{unitary Cayley graphs} on $\Z_n$, which are defined by $X_n = \Cay(\Z_n, \Z_n\units)$.
These form the basis of the subject of graph representations~\cite{EE89}, and are also studied as objects of independent interest: see for example~\cite{DG95, BG04, KS07, RV09}.

We consider a subgraph $G_n \le X_n$ of the unitary Cayley graphs, defined as follows.
Let $Q_n = \ens{u^2 \,\big|\, u \in \Z_n\units}$ be the group of \emph{quadratic units} modulo $n$ (quadratic residues which are also multiplicative units), and $T_n = \pm Q_n$.
We then define $G_n = \Cay(\Z_n, T_n)$, in which two vertices given by $a,b \in \Z_n$ are adjacent if and only if their difference is a quadratic unit in $\Z_n$, \ie\ if $a - b \in \ens{\pm u^2 \,\big|\, u \in \Z_n\units}$.
In the case where $n \equiv 1 \pmod{4}$ and is prime, $G_n$ coincides with the Paley graph on $n$ vertices: thus the graphs $G_n$ are a circulant generalization of these graphs for arbitrary $n$.
We refer to $G_n$ as the \emph{(undirected) quadratic unitary Cayley graph} on $\Z_n$.

We present some structural properties of quadratic unitary Cayley graphs $G_n$.
In particular, we characterize its decompositions into tensor products over relatively prime factors of $n$, and categorize the graphs $G_n$ in terms of their diameters.
From these results, we obtain a corollary regarding the decomposition of symplectic matrices $S \in \Sp_{2m}(\Z_n)$ in terms of symplectic row-operations, consisting of symplectic scalar multiplications, symplectic row-swaps, and symplectic row-additions/subtractions.
We also characterize the conditions under which quadratic unitary graphs are perfect, by examining special cases of quadratic unitary graphs which are self-complementary.

\begin{notation}
  Throughout the following, $n = p_1^{m_1} p_2^{m_2} \cdots p_t^{m_t}$ is a decomposition of $n$ into powers of distinct primes, and $\sigma: \Z_n \to* \Z_{\smash{p_1}^{\!m_1}} \oplus \cdots \oplus \Z_{\smash{p_t}^{\!m_t}}$ is the isomorphism of rings which is induced by the Chinese Remainder theorem.
  (We refer to similar isomorphisms $\rho: \Z_n \to \Z_M \oplus \Z_N$ for coprime $M$ and $N$ as \emph{natural} isomorphisms.)
  We sometimes describe the properties of $G_n$ in terms of the directed Cayley graph $\Gamma_n = \Cay(\Z_n, Q_n)$, whose arcs $a \to* b$ correspond to addition (but not subtraction) of a quadratic unit to a modulus $a \in \Z_n$; we may refer to this as the \emph{directed quadratic unitary Cayley graph}.
\end{notation}

\section{Tensor product structure}

By the isomorphism $\Z_n\units \cong \Z\units_{\smash{p_1}^{\!m_1}} \oplus \cdots \oplus \Z_{\smash{p_t}^{\!m_t}}\units$ induced by $\sigma$, unitary Cayley graphs $X_n$ may be decomposed as \emph{tensor products} $X_n \cong X_{\smash{p_1}^{\!m_1}} \ox \cdots \ox X_{\smash{p_t}^{\!m_t}}$ of smaller unitary Cayley graphs (also called direct products~\cite{RV09} or Kronecker products~\cite{W62}, among other terms):

\begin{definition}
	\label{def:tensorProduct}
	The \emph{tensor product} $A \ox B$ of two \mbox{(di-)graphs} $A$ and $B$ is the \mbox{(di-)graph} with vertex-set $V(A) \x V(B)$, where $((u,u'), (v,v')) \in E(A \ox B)$ if and only if $((u,v), (u',v')) \in E(A) \x E(B)$.\footnote{%
		We write $A_1 \ox (A_2 \ox A_3) = (A_1 \ox A_2) \ox A_3 = A_1 \ox A_2 \ox A_3$, and so on for higher-order tensor products,
		similarly to the convention for Cartesian products of sets.
	}
\end{definition}
\noindent 
Corollary~3.3 of \cite{RV09} gives an explicit proof that $X_n \cong X_{\smash{p_1}^{\!m_1}} \ox \cdots \ox X_{\smash{p_t}^{\!m_t}}$; a similar approach may be used to decompose any \mbox{(di-)graph} $\Cay(R,M)$ for rings $R = R_1 \oplus \cdots \oplus R_t$ and multiplicative monoids $M_1 = M_1 \oplus \cdots \oplus M_t$ where $M_j \subset R_j$.
For instance, as $Q_n \cong Q_{\smash{p_1}^{\!m_1}} \oplus \cdots \oplus Q_{\smash{p_t}^{\!m_t}}$\,, it follows that $\Gamma_n \cong \Gamma_{\smash{p_1}^{\!m_1}} \ox \cdots \ox \Gamma_{\smash{p_t}^{\!m_t}}$ as well.

It is reasonable to suppose that the graphs $G_n$ will also exhibit tensor product structure; however, they do not always decompose over the prime power factors of $n$ as do $X_n$ and $\Gamma_n$.
This is because $T_n$ may fail to decompose as a direct product of groups over the prime-power factors $p_j^{m_j}$.
By definition, for each $j$, we either have $T_{\smash{p_j}^{\!m_j}} = Q_{\smash{p_j}^{\!m_j}}$ or $T_{\smash{p_j}^{\!m_j}} \cong Q_{\smash{p_j}^{\!m_j}} \oplus \gen{-1}$; when $Q_{\smash{p_j}^{\!m_j}} < T_{\smash{p_j}^{\!m_j}}$ for multiple $p_j$, one cannot decompose $T_n$ over the prime-power factors of $n$.
We may generalize this observation as follows:

\begin{theorem}
	\label{thm:tensorProductFactorizable}
	For coprime integers $M, N \ge 1$, we have $G_M \ox G_N \cong G_{MN}$ if and only if either $-1 \in Q_M$ or $-1 \in Q_N$.
\end{theorem}
\begin{proof}
	We have $G_M \ox G_N \cong G_{MN}$ if and only if $T_M \oplus T_N \cong T_{MN}$.
	Let $\rho: \Z_{MN} \to \Z_M \oplus \Z_N$ be the natural isomorphism: this induces an isomorphism $Q_{MN} \cong Q_M \oplus Q_N$, and will also induce an isomorphism $T_{MN} \cong T_M \oplus T_N$ if the two groups are indeed isomorphic.
	Clearly, $\sigma(T_{MN}) \le T_M \oplus T_N$; we consider the opposite inclusion.

	If $-1 \notin Q_M$ and $-1 \notin Q_N$, we have $(-1,\,1),(1,-1) \notin Q_M \oplus Q_N$; as both tuples are elements of $T_M \oplus T_N$, but neither of them are elements of $\pm(Q_M \oplus Q_N) = \sigma(\pm Q_{MN}) = \sigma(T_{MN})$, it follows that $T_{MN}$ and $T_M \oplus T_N$ are not isomorphic in this case.
	Conversely, consider $u \in \Z_n\units$ arbitrary, and let $(u_M, u_N) = \rho(u)$.
	If $-1 \in Q_M$, let $i \in \Z_M$ such that $i^2 = -1$: for any $s_M,s_N \in \ens{0,1}$, we then have
	\begin{align}
			\Big((-1)^{s_M} u_M^2,\, (-1)^{s_N} u_N^2\Big)
		\;=&\;\,
			(-1)^{s_N} \Big((-1)^{s_M - s_N} u_M^2,\, u_N^2\Big)
		\notag\\=&\;\,
			(-1)^{s_N} \Big( \big[i^{(s_M - s_N)} u_M\big]^2,\, u_N^2\Big)
		\,.
	\end{align}
	Thus $T_M \oplus T_N \le \sigma(T_{MN})$; and similarly if $-1 \in Q_N$.
\end{proof}

\paragraph{Remark.} The above result is similar to~\cite[Theorem~8]{KS09}, which uses a ``partial transpose'' criterion to indicate when a graph may be regarded as a symmetric difference of tensor products of graphs on $M$ and $N$ vertices; the presence of $-1$ in either $Q_M$ or $Q_N$ is equivalent to $G_{MN}$ being invariant under partial transposes (w.{}r.{}t.{} to the tensor decomposition induced by $\rho$).

\begin{corollary}
	\label{cor:decomposeGnTensors}
	For $n \ge 1$, let $n = p_1^{m_1} \cdots p_\tau^{m_\tau} N$ be a factorization of $n$ such that $p_j \equiv 1 \pmod{4}$ for each $1 \le j \le \tau$, and $N$ has no such prime factors.
	Then $G_n \cong G_{\smash{p_1}^{\!m_1}} \ox \cdots \ox G_{\smash{p_\tau}^{\!m_\tau}} \ox G_N$\,.
\end{corollary}
\begin{proof}
	For $p_j$ odd, $\Z_{\smash{p_j}^{\!m_j}}\units$ is a cyclic group~\cite{Gauss} of order $(p_j - 1)p_j^{m_j - 1}$ in which $-1$ is the unique element of order two: then $-1$ is a quadratic residue modulo $p_j^{m_j}$ if  and only if $p_j \equiv 1 \pmod{4}$.
	As this holds for all $1 \le j \le \tau$, repeated application of Theorem~\ref{thm:tensorProductFactorizable} yields the decomposition above.
\end{proof}

\begin{corollary}
	For $n \ge 1$, we have $G_n \cong G_{\smash{p_1}^{\!m_1}} \ox \cdots \ox G_{\smash{p_t}^{\!m_t}}$ if and only if either $n$ has at most one prime factor $p_j \not\equiv 1 \pmod{4}$, or $n$ has two such factors and $n \equiv 2 \pmod{4}$.
\end{corollary}
\begin{proof}
	Suppose that $G_n$ decomposes as above.
	Let $N$ be the largest factor of $n$ which does not have prime factors $p \equiv 1 \pmod{4}$: we continue from the proof of Corollary~\ref{cor:decomposeGnTensors}.
	By Theorem~\ref{thm:tensorProductFactorizable}, $G_N$ itself decomposes as a tensor factor over its prime power factors $p_{\tau+1}^{m_{\tau+1}}, \ldots, p_t^{m_t}$ if and only if there is at most one such prime $p_j$ such that $-1 \notin Q_{\smash{p_j}^{\!m_j}}$.
	However, by construction, all odd prime factors $p_j$ of $N$ satisfy $p_j \equiv 3 \pmod{4}$, in which case $-1 \notin Q_{\smash{p_j}^{\!m_j}}$ for any of them.
	Furthermore, for $m \ge 2$, we have $r \in Q_{2^m}$ only if $r \equiv 1 \pmod{4}$; then $-1 \in Q_{2^m}$ if and only if $2^m = 2$.
	Thus, if $G_n \cong G_{\smash{p_1}^{\!m_1}} \ox \cdots \ox G_{\smash{p_t}^{\!m_t}}$, it follows either that $N = p^m$ for some prime $p \equiv 3 \pmod{4}$, in which case the decomposition of Corollary~\ref{cor:decomposeGnTensors} is the desired decomposition, or $N = 2 p^m$ for some prime $p \equiv 3 \pmod{4}$, in which case $n \equiv 2 \pmod{4}$.
	The converse follows easily from Corollary~\ref{cor:decomposeGnTensors} and Theorem~\ref{thm:tensorProductFactorizable}.
\end{proof}

We finish our discussion of tensor products with an observation for prime powers.
Let $\mathring{K}_M$ denote the complete pseudograph on $M$ vertices (\ie\ an $M$-clique with loops):
\begin{lemma}
	\label{lemma:primePowerTensorDecomp}
	For $m \ge 3$, we have $G_{2^m\!} \cong G_8 \ox \mathring K_{2^{m-3}\!\!}$ \,and\, $\Gamma_{2^m\!} \cong \Gamma_8 \ox \mathring K_{2^{m-3}}$; for $p$ an odd prime and $m \ge 1$, we have $G_{p^m\!} \cong G_p \ox \mathring K_{p^{m-1}\!\!}$ \,and\, $\Gamma_{p^m\!} \cong \Gamma_p \ox \mathring K_{p^{m-1}}$.
\end{lemma}
\begin{proof}
	We prove the results for $\Gamma_{p^m}$; the results for $G_{p^m}$ are similar.
	\begin{itemize}	
 	\item
		Let $n = 2^m$ for $m \ge 3$.
		We have $q \in Q_n$ if and only if $q \equiv 1 \pmod{8}$.
		Let $\tau: \Z_{2^m} \to* \Z_8 \x \Z_{2^{m-3}}$ (not a ring homomorphism) be defined by $\tau(r) = (r', k')$ such that $r = 8k' + r'$ for $r' \in \ens{0, \ldots, 7}$.
		Then, we have $a - b \in Q_n$ if and only if $\tau(a - b) \in \ens{1} \x \Z_{2^{m-3}}$, so that $\tau$ induces a homomorphism $\Gamma_n \cong \Gamma_8 \ox \mathring K_{2^{m-3}}$.

	\item
		Similarly, for $n = p^m$ for $p$ an odd prime and $m \ge 1$, we have $q = pk' + q' \in Q_n$ (for $q' \in \ens{0, \ldots, p-1}$, which we we identify with $\Z_p$) if and only if $q' \in Q_p$.
		If $\tau: \Z_{2^m} \to* \Z_p \x \Z_{p^{m-1}}$ is defined by $\tau(q) = (q', k')$, we then have $a - b \in Q_n$ if and only if $\tau(a - b) \in Q_p \x \Z_{p^{m-1}}$.
		Thus, $\tau$ induces a homomorphism $\Gamma_n \cong \Gamma_p \ox \mathring K_{p^{m-1}}$. \qedhere
 	\end{itemize}
\end{proof}
\noindent
Together with Corollary~\ref{cor:decomposeGnTensors}, and the fact that $\mathring K_{p^m}$ itself may be decomposed for any prime $p$ as an $m$-fold tensor product $\mathring K_p \ox \cdots \ox \mathring K_p$, the graph $G_n$ may be decomposed very finely whenever $n$ is dominated by prime-power factors $p^m$ for $p \equiv 1 \pmod{4}$.

\section{Induced paths and cycles of $G_n$}
\label{sec:pathsCycles}

Even when the graph $G_n$ does not itself decompose as a tensor product, we may fruitfully describe such properties as walks in the graphs $G_n$ in terms of correlated transitions in tensor-factor ``subsystems''.
This intuition will guide the analysis of this section in our characterization both of the diameters of the graphs $G_n$, and of the factors of $n$ for $G_n$ a perfect graph.

As $T_n$ is a multiplicative subgroup of $\Z_n\units$, we may easily show that the graphs $G_n$ are arc-transitive.
For any pair of edges $vw, v'w' \in E(G_n)$, the affine function $f(x) = (w'-v')(w-v)\inv(x-v) + v'$ is an automorphism of $G_n$ which maps $v \mapsto* v'$ and $w \mapsto* w'$.
Consequently $G_n$ is vertex-transitive as well, so that we may bound the diameter by bounding the distance of vertices $v \in V(G)$ from $0 \in V(G)$, and also restrict our attention to odd induced cycles (or \emph{odd holes}) which include $0$ in our analysis of perfect graphs.

Let $A_n$, $B_n$ be the adjacency graphs of the graph $G_n$ and the digraph $\Gamma_n$ respectively.
We then have $A_n = B_n = B_n\trans$ if and only if $-1$ is a quadratic residue modulo $n$, and $A_n = B_n + B_n\trans$ otherwise; in either case, we have $A_n \,\propto\, B_n + B_n\trans$.
As $B_n$ may be decomposed as a Kronecker product (corresponding to the tensor decomposition of $\Gamma_n$), this suggests an analysis of walks in $G_n$ in terms of ``synchronized walks'' in the rings $\Z_{\smash{p_j}^{\!m_j}}$ by adding or subtracting quadratic units, where one must add a quadratic unit in all rings simultaneously or subtract a quadratic unit in all rings simultaneously.
This will inform the analysis of properties such as the diameters and perfectness of the graphs $G_n$.

\subsection{Characterizing paths of length two for $n$ odd}

To facilitate the analysis of this section, we will be interested in enumerating paths of length two in $G_n$ between distinct vertices.
Because $A_n \,\propto\, B_n + B_n\trans$ for all $n$, we have
\begin{align}
		A_n^2
	\;\;\propto&\;\;
		B_n^2
	\,+\,
		2 B_n B_n\trans
	\,+\,
		\big( B_n \trans \big)^2
	\notag\\[0.5ex]\cong&\;\;
		\sqparen{\bigotimes_{j = 1}^t B_{\smash{p_j}^{\!m_j}}^2}
		\;+\;
		2 \sqparen{\bigotimes_{j = 1}^t B_{\smash{p_j}^{\!m_j}} B_{\smash{p_j}^{\!m_j}}\trans}
		\;+\;
		\sqparen{\bigotimes_{j = 1}^t \big( B_{\smash{p_j}^{\!m_j}} \trans\big)^2}\,,
\end{align}
where congruence is up to a permutation of the standard basis.
Thus, we may characterize the paths of length two in $G_n$ between distinct vertices $r, s \in \Z_n$ in terms of the number of ways that we may represent $s - r$ in the form $\alpha^2 + \beta^2$, $\alpha^2 - \beta^2$, and $-\alpha^2 - \beta^2$ for some units $\alpha,\beta \in \Z_n\units$; and these we may characterize in terms of products over the number of representations in the special case where $n$ is a prime power.

\begin{definition}
	For $n > 0$ and $r \in \Z_n$, we let $S_n(r)$ denote the number of solutions $(x,y) \in Q_n \x Q_n$ to the equation $r \,=\, x + y$; similarly, $D_n(r)$ denotes the number of solutions $(x,y) \in Q_n \x Q_n$ to the equation $r \;=\; x - y$.
\end{definition}
\noindent 
Thus, when $-1 \in Q_n$ and $A_n = B_n = \frac{1}{2}(B_n + B_n\trans)$, the number of paths of length two from $0$ to $r \ne 0$ is $S_n(r)$; otherwise, if $-1 \notin Q_n$, the number of such paths is $S_n(r) + 2D_n(r) + S_n(-r)$. 
Thus, the number of paths of length two from $0$ to $r$ reduces to avaluation of the functions $S_n$ and $D_n$.
We may evaluate these functions for $n$ a prime power, through a straightforward generalization of standard results on patterns of quadratic residues and non-residues to prime power moduli:
\begin{lemma}
	\label{lemma:consecutiveResidues}
	For $p$ a prime and $m \ge 1$, let $C^{++}_{p^m}$ (respectively $C^{--}_{p^m}$) denote the number of consecutive pairs of quadratic units (resp. consecutive pairs of non-quadratic units) modulo $p^m$, and $C^{+-}_{p^m}$ (respectively $C^{-+}_{p^m}$) denote the number of sequences of a quadratic unit followed by a non-quadratic unit (resp. a non-quadratic unit followed by a quadratic unit) modulo $p^m$.
	For primes $p \equiv 1 \pmod{4}$, we have
	\begin{subequations}
	\begin{align}
		C^{++}_{p^m}
	  =&\;
		\frac{(p-5)p^{m-1}}{4}	\;,
	&
		C^{+-}_p =\, C^{-+}_p =\, C^{--}_p
	  =&\;
		\frac{(p-1)p^{m-1}}{4}	\;;
	\end{align}
	otherwise, if $p \equiv 3 \pmod{4}$, we have
	\begin{align}
		C^{+-}_p
	  =&\;
		\frac{(p+1)p^{m-1}}{4}	\;,
	&
		C^{++}_p =\, C^{-+}_p =\, C^{--}_p
	  =&\;
		\frac{(p-3)p^{m-1}}{4}	\;.
	\end{align}
	\end{subequations}
\end{lemma}
\begin{proof}
	As $r \in \Z$ is a quadratic residue, quadratic non-residue, and/or unit modulo $p^m$ if and only the same properties hold modulo $p$, the distribution of quadratic and non-quadratic units modulo $p^m$ is simply that of the integers modulo $p$, repeated $p^{m-1}$ times.
	It then suffices to multiply the formulae given for $C_p^{++}$, $C_p^{+-}$, $C_p^{-+}$, $C_p^{--}$ (obtained by Aladov~\cite{Aladov1896}) by $p^{m-1}$.
\end{proof}
\begin{lemma}
	\label{lemma:countingSumSquares}
	Let $p$ be an odd prime, $m > 0$, and $r \in \Z_{p^m}$.
	If $p \equiv 1 \pmod{4}$, we have
	\begin{subequations}
	\label{eqn:sumOfSquares}
	\begin{align}
 		\mspace{-5mu}
	 		S_{p^m}(r)
		\,=\,
			D_{p^m}(r)
		\,=&\;
			\begin{cases}
				\frac{1}{4} (p-5) p^{m-1} 	\;,	&	\text{for $r$ a quadratic unit},\!			\\[1ex]
				\frac{1}{4}	(p-1) p^{m-1}	\;,	&	\text{for $r$ a non-quadratic unit},\!		\\[1ex]
				\frac{1}{2}(p-1)p^{m-1}		\;,	&	\text{for $r$ a zero divisor};\!
			\end{cases}
	\intertext{for $p \equiv 3 \!\!\!\pmod{4}$, we instead have}
 		\mspace{-5mu}
	 		S_{p^m}(r)
		\,=&\;
			\begin{cases}
				\frac{1}{4}	(p-3) p^{m-1} 	\;,	&	\text{for $r$ a quadratic unit},\!		\\[1ex]
				\frac{1}{4}	(p+1) p^{m-1}	\;,	&	\text{for $r$ a non-quadratic unit},\!	\\[1ex]
				0									\;,	&	\text{for $r$ a zero divisor};\!
			\end{cases}
	\\[1em]
 		\mspace{-5mu}
	 		D_{p^m}(r)
		\,=&\;
			\begin{cases}
				\frac{1}{4}	(p-3) p^{m-1} 		\;,	&	\text{for $r$ a unit},\!					\\[1ex]
				\frac{1}{2} (p-1) p^{m-1}		\;,	&	\text{for $r$ a zero divisor}.\!
			\end{cases}
	\end{align}
	\end{subequations}
\end{lemma}
\begin{proof}
	We proceed by cases, according to whether $r$ is a quadratic unit, non-quadratic unit, or zero modulo $p$:
	\begin{itemize}
	\item
	  Suppose $r \in Q_n$.
	  Each consecutive pair $q,q+1 \in Q_{p^m}$ yields a solution $(x,y) = (r(q+1),rq) \in Q_{p^m} \x Q_{p^m}$ to $x - y = r$; then we have $D_{p^m}(r) = C_{p^m}^{++}$.
	  Similarly, each such pair yields a solution $(x,y) = (rq(q+1)\inv, r(q+1)\inv) \in Q_{p^m} \x Q_{p^m}$ to $x + y = r$; then $S_{p^m}(r) = C_{p^m}^{++}$ as well.

	\item
	  Suppose $r \in \Z_{p^m}\units \setminus Q_{p^m}$.
	  Each consecutive pair $s,s+1 \in \Z_{p^m}\units \setminus Q_{p^m}$ represents a solution in non-quadratic units to $x - y = 1$; these may then be used to obtain solutions $(rx, ry) \in Q_{p^m} \x Q_{p^m}$ to $rx - ry = r$, so that $D_{p^m}(r) = C_{p^m}^{--}$.
	  In the case that $p \equiv 1 \pmod{4}$, the negation of a quadratic unit is also a quadratic unit; in this case, we have the same number of solutions $(rx, -ry) \in Q_{p^m} \x Q_{p^m}$ to $rx + (-ry) = r$, so that $S_{p^m}(r) = C_p^{--}$ as well.

	  If instead $p \equiv 3 \pmod{4}$, we instead consider quadratic units $s \in Q_{p^m}$ such that $s+1$ is a non-quadratic unit.
	  Each such pair yields a solution $(x,y) = (r(s+1), -rs) \in Q_{p^m} \x Q_{p^m}$ to $x + y = r$; then we have a solution for each such pair $s,s+1$, so that $S_{p^m}(r) = C_{p^m}^{+-}$.

	\item
	  Finally, suppose $r$ is a multiple of $p$.
	  The congruence $x + y \equiv 0 \pmod{p}$ is satisfiable for $(x,y) \in Q_{p^m} \x Q_{p^m}$ only if $-x$ is a quadratic unit modulo $p$ for some $x \in Q_{p^m}$, \ie\ if $p \equiv 1 \pmod{4}$.
	  If this is the case, then every $x \in Q_{p^m}$ contributes a solution $(x,y) = (x, r - x) \in Q_{p^m} \x Q_{p^m}$ to $x + y = r$; otherwise, in the case $p \equiv 3 \pmod{4}$, there are no solutions.
	  Similarly, regardless of the value of $p$, each quadratic unit $x \in Q_{p^m}$ contributes a solution $(x,y) = (x, x - r) \in Q_{p^m} \x Q_{p^m}$ to $x - y = r$.
	  Thus $D_{p^m}(r) = \frac{1}{2}(p-1)$ for all $p$; $S_{p^m}(r) = \frac{1}{2}(p-1)$ for $p \equiv 1 \pmod{4}$; and $S_{p^m}(r) = 0$ for $p \equiv 3 \pmod{4}$.
	  \qedhere
	\end{itemize}
\end{proof}

\begin{corollary}
	\label{cor:diameterOddPrimePower}
	$\diam(G_{p^m}) \le 2$ for $p$ an odd prime and $m > 0$; this inequality is strict if and only if $p \equiv 3 \pmod 4$ and $m = 1$.
\end{corollary}
\begin{proof}
	Clearly for $p \equiv 1 \pmod{4}$ we have $\diam(G_{p^m}) = 2$; suppose then that $p \equiv 3 \pmod 4$.
	We may form any zero divisor $s = pk$ as a difference of quadratic units $x \in Q_{p^m}$ and $x - pk \in Q_{p^m}$, so that $\diam(G_{p^m}) \le 2$.
	We have $\diam(G_{p^m}) = 1$ only if $0$ is the only zero divisor of $\Z_{p^m}$; this implies that $m = 1$, in which case $T_{p^m} = \Z_p\units$, so that the converse also holds.
\end{proof}

In Lemma~\ref{lemma:countingSumSquares}, $n = 3^m$ and $n = 5^m$ are cases for which there do not exist paths of length two from zero to any quadratic unit.
This does not affect the diameters of the graphs $G_{3^m}$ or $G_{5^m}$ for $m > 0$; however, using the following Lemma, we shall see that this deficiency affects the diameters of $G_n$ for any other $n$ a multiple of either $3$ or $5$.

\begin{lemma}
	\label{lemma:oddPathsLength2}
	For $n > 0$ odd and $r \in \Z_n$, we have $S_n(r) = 0$ if and only if at least one of the following conditions hold:
	\begin{romanum}
	\item
		$n$ is a multiple of $3$, and $r \not\equiv 2 \pmod{3}$;

	\item
		$n$ is a multiple of $5$, and $r \equiv \pm 1 \pmod{5}$; or

	\item
		$n$ has a prime factor $p_j \equiv 3 \pmod{4}$ such that $r \in  p_j \Z_n$.
	\end{romanum}
	Similarly, we have $D_n(r) = 0$ if and only if at least one of the following conditions hold:
	\begin{romanum}
	\item
		$n$ is a multiple of $3$, and $r \not\equiv 0 \pmod{3}$; or

	\item
		$n$ is a multiple of $5$, and $r \equiv \pm 1 \pmod{5}$.
	\end{romanum}
\end{lemma}
\begin{proof}
	For $r \in \Z_n$ arbitrary, let $(r_1, r_2, \ldots, r_t) = \sigma(r)$.
	By the decompositions $B^2_n \,\cong\, B^2_{\smash{p_1}^{\!m_1}} \ox \cdots \ox B^2_{\smash{p_t}^{\!m_t}}$ and $B_n B_n\trans \,\cong\, B_{\smash{p_1}^{\!m_1}}B_{\smash{p_1}^{\!m_1}}\trans \ox \cdots \ox B_{\smash{p_t}^{\!m_t}}B_{\smash{p_t}^{\!m_t}}\trans$, we may express $S_n(r)$ and $D_n(r)$ as products over the prime-power factors of $n$,
	\begin{align}
		 	S_n(r)
		\;=&\;
			\prod_{j = 1}^t S_{\smash{p_j}^{\!m_j}}(r_j)	\;,
		&
		 	D_n(r)
		\;=&\;
			\prod_{j = 1}^t D_{\smash{p_j}^{\!m_j}}(r_j)	\;.
	\end{align}
	These are zero if and only if there exist $1 \le j \le t$ such that $S_{\smash{p_j}^{\!m_j}}(r_j) = 0$ or $D_{\smash{p_j}^{\!m_j}}(r_j) = 0$, respectively.
	By Lemma~\ref{lemma:countingSumSquares}, $S_{\smash{p_j}^{\!m_j}}(r_j) = 0$ if and only if either $r_j$ is a zero divisor of $\Z_{\smash{p_j}^{\!m_j}}$ for a prime factor $p_j \equiv 3 \pmod{4}$, or if $p_j \in \ens{3,5}$ and $r_j$ is a quadratic unit modulo $\Z_{\smash{p_j}^{\!m_j}}$\,; similarly, $D_{\smash{p_j}^{\!m_j}}(r_j) = 0$ if and only if $p_j = 3$ and $r_j$ is a unit modulo $3$, or $p_j = 5$ and $r_j$ is a quadratic unit modulo $5$.
\end{proof}

\subsection{Diameter of $G_n$ for odd $n$}

For odd integers $n$, characterizing the diameters of $G_n$ involves accounting for ``problematic'' prime factors of $n$ (those described in Lemma~\ref{lemma:oddPathsLength2}), which present obstacles to the construction of short paths between distinct vertices:

\begin{theorem}
	\label{thm:oddDiameter}
	Let $n > 1$ odd.
	Let $\gamma_3(n) = 1$ if $n$ is a multiple of $3$, and $\gamma_3(n) = 0$ otherwise; $\delta_3(n) = 1$ if $n$ has prime factors $p_j \equiv 3 \pmod{4}$ for $p_j > 3$, and $\delta_3(n) = 0$ otherwise; and $\gamma_5(n) = 1$ if $n$ is a multiple of $5$, and $\gamma_5(n) = 0$ otherwise.
	Then, we have
	\begin{align*}
			\diam(G_n)
		\,=&\;
			\begin{cases}
				1	,	&	\!\!\text{if $n$ is prime and $n \equiv 3 \!\!\!\!\pmod{4}$};			\\[0.5ex]
				2	,	&	\!\!\text{if $n$ is prime and $n \equiv 1 \!\!\!\!\pmod{4}$};			\\[0.5ex]
				2	,	&	\!\!\text{if $\omega(n) = 1$ and $n$ is composite};						\\[0.5ex]
				2 + \gamma_3(n) \delta_3(n) + \gamma_5(n) ,	&	\!\!\text{if $\omega(n) > 1$}.
% 
% 
% 				2	\;,	&	\text{if $\omega(n) > 1$ and $n$ is relatively prime to $15$};											\\[0.5ex]
% 				2	\;,	&	\text{if $\omega(n) > \omega_3(n) = 1$, and $n$ is a multiple of $3$ coprime to $5$};				\\[0.5ex] 
% 				2	\;,	&	\text{if $\omega_3(n) \le 1 < \omega(n)$ and $n$ is coprime to $5$};	\\[0.5ex]
% 				3	\;,	&	\text{if $\omega_3(n) > 1$ or $n$ a multiple of $15$, but not both};					\\[0.5ex]
% 				4	\;,	&	\text{if $\omega_3(n) > 1$ and $n$ is a multiple of $15$}.
			\end{cases}
	\end{align*}
	In particular, $\diam(G_n) \le 4$.
\end{theorem}
\begin{proof}
	The diameters for $\omega(n) = 1$ are characterized by Corollary~\ref{cor:diameterOddPrimePower}: we thus restrict ourselves to the case $\omega(n) > 1$.

	We have $\diam(G_n) \le 2$ if and only if either $S_n(r)$, $S_n(-r)$, or $D_n(r)$ is positive for all $r \in \Z_n \setminus T_n$.
	By Lemma~\ref{lemma:oddPathsLength2}, $D_n(r) > 0$ for all $r \in \Z_n$ if $n$ is relatively prime to $15$; then $\diam(G_n) = 2$, and $r = u - u'$ for some $u, u' \in Q_n$ for any $r \in \Z_n$ if $\gamma_3 = \gamma_5 = 0$.
	If $n$ is a multiple of $5$, however, we have $S_n(r) = S_n(-r) = D_n(r) = 0$ for any non-quadratic unit $r \equiv \pm 1 \pmod{5}$, of which there is at least one (as $n$ is not a power of $5$): thus $\diam(G_n) \ge 3$ if $\gamma_5(n) = 1$.

	Suppose that $n$ is relatively prime to $5$, and is a multiple of $3$.
	Again by Lemma~\ref{lemma:oddPathsLength2}, there are walks of length two from $0$ to $r$ if $r \equiv 0 \pmod{3}$, as we have $D_n(r) > 0$ in this case.
	However, if $n$ has prime factors $p_j > 3$ such that $p_j \equiv 3 \pmod{4}$, there exist $r \in p_j \Z_n$ such that $r \not\equiv 0 \pmod{3}$, in which case we have $S_n(r) = S_n(-r) = D_n(r) = 0$.
	Thus, if $\gamma_3(n) = \delta_3(n) = 1$, we have $\diam(G_n) \ge 3$.
	Otherwise, if $\delta_3(n) = 0$, we have either $S_n(r) > 0$ in the case that $r \equiv 2 \pmod{3}$, or $S_n(-r) > 0$ in the case that $r \equiv 1 \pmod{3}$.
	In this case, every vertex $r \ne 0$ is reachable by a path of length two, so that $\diam(G_n) = 2$ if $\gamma_3(n) = 1$ and $\delta_3(n) = \gamma_5(n) = 0$.

	Finally, suppose that either $\gamma_5(n) = 1$ or $\gamma_3(n) = \delta_3(n) = 1$: from the analysis above, we have $\diam(G_n) \ge 3$.
	For $r \in \Z_n$, let $(r_1, \ldots, r_n) = \sigma(r)$, where we arbitrarily label $p_3 = 3$ if $n$ is a multiple of $3$, and $p_5 = 5$ if $n$ is a multiple of $5$.
	We may then classify the distance of $r \in V(G_n)$ away from zero, as follows.
	\begin{itemize}
	\item
		Suppose that $n$ is a multiple of $3$ and some other $p_j \equiv 3 \pmod{4}$, and that either $n$ is relatively prime to $5$ or $r \not\equiv \pm1 \pmod{5}$.
		By Lemma~\ref{lemma:oddPathsLength2}, we have $D_n(r) > 0$ if $r \equiv 0 \pmod{3}$, in which case it is at a distance of two from $0$.
		Otherwise, for $r \equiv \pm 1 \pmod{3}$, let $s = r \mp u$ for $u \in Q_n$: then $s \equiv 0 \pmod{3}$.
		Then $D_n(s) > 0$, in which case $r = u'' - u' \pm u$ for some choice of units $u', u'' \in Q_n$, so that $r$ can be reached from $0$ by a walk of length three.

	\item 
		Suppose that $n$ is a multiple of $5$ and that $r \not\equiv 0 \pmod{5}$.
		We may select coefficients $u_j \in Q_{\smash{p_j}^{\!m_j}}$ such that $r_5 - u_5 \in \ens{2,3}$, and such that $u_j \ne r_j$ for any $p_j \ge 7$.
		Let $u = \sigma\inv(u_1, \ldots, u_t)$: by construction, we then have $r - u \equiv \pm 2 \pmod{5}$ and $r - u \not\equiv 0 \pmod{p_j}$ for $p_j \ge 7$.
		Then either $S_n(r - u) > 0$, $S_n(u - r) > 0$, or $D_n(r - u) > 0$ (according to whether or not $n$ is a multiple of $3$, and which residue $r$ has modulo $3$ if so): $r$ can then be reached from $0$ by a path of length three.

\item
	Suppose that $n$ is a multiple of $5$, and that $r \equiv 0 \pmod{5}$.
	If $n$ is not a multiple of $3$, or if $r \equiv 0 \pmod{3}$, then $D_n(r) > 0$; $r$ can then be reached from $0$ by a walk of length two.
	We may then suppose that $n$ is a multiple of $3$ and $r \equiv \pm 1 \pmod{3}$.
	If we also have $r \not\equiv 0 \pmod{p_j}$ for any $p_j \equiv 3 \pmod{4}$, one of $S_n(r)$ or $S_n(-r)$ is non-zero; again, $r$ is at a distance of two from $0$.
	Otherwise, we have $r \equiv 0 \pmod{p_j}$ for any $p_j \equiv 3 \pmod{4}$, so that $S_n(r) = S_n(-r) = D_n(r) = 0$; then $r$ has a distance at least three from $0$.
	As well, any neighbor $s = r \pm u$ (for $u \in Q_n$ arbitrary) satsifies $s \equiv \pm 1 \pmod{5}$.
	Then each neighbor of $r$ is then at distance three from $0$ in $G_n$, from which it follows that $r$ is at a distance of four from $0$.
\end{itemize}
Thus, there exist vertices at distance four from $0$ if $\gamma_3(n)\delta_3(n) + \gamma_5(n) = 2$; and apart from these vertices, or in the case that $\gamma_3(n)\delta_3(n) + \gamma_5(n) = 1$, each vertex is at a distance of at most three from $0$.
Then $\diam(G_n) = 2 + \gamma_3(n)\delta_3(n) + \gamma_5(n)$ if $\omega(n) > 1$, as required.
\end{proof}

\subsection{Restricted reachability results for $n$ coprime to $6$}

We may prove some stronger results on the reachability of vertices from $0$ in $G_n$ for $n$ odd: this will facilitate the analysis of perfectness results and the diameters for $n$ even.

\begin{definition}
	\label{def:uniformDiameter}
	For a \mbox{(di-)graph} $G$, the \emph{uniform diameter} $\udiam(G)$ is the minimum integer $d$ such that, for any two vertices $v,w \in V(G)$, there exists a (directed) walk of length $d$ from $v$ to $w$ in $G$.
\end{definition}
\noindent
Our interest in ``uniform'' diameters is due to the fact that if every vertex $v \in V(\Gamma_n)$ can be reached from $0$ by a path of exactly $d$ in $\Gamma_n$, then $v$ can also be reached from $0$ by a path of any length $\ell \ge d$ as well, which will prove useful for describing walks in $\Gamma_n$ to arbitrary vertices in terms of simultaneous walks in the digraphs $\Gamma_{\smash{p_j}^{\!m_j}}$.

We may easily show that $\Gamma_n$ has no uniform diameter when $n$ is a multiple of $3$.
For any adjacent vertices $v$ and $w$ such that $w - v \in Q_n$, we have $w - v \equiv 1 \pmod{3}$ by that very fact.
Then, there is a walk of length $\ell$ from $v$ to $w$ only if $\ell \equiv 1 \pmod{3}$; similarly, there is a walk of length $\ell$ from $w$ to $v$ only if $\ell \equiv 2 \pmod{3}$.
For similar reasons, $\Gamma_n$ has no uniform diameter for $n$ even.
However, for $n$ relatively prime to $6$, $\Gamma_n$ has a uniform diameter which may be easily characterized:
\begin{theorem}
	\label{thm:oddUniformDiameter}
	Let $n = p_1^{m_1} \cdots p_t^{m_t}$ be relatively prime to $6$.
	Then 
	\begin{align*}
		 	\udiam(\Gamma_n)
		\;=\;
			\begin{cases}
				2	\;,	&	\text{if $n$ is coprime to $5$ and $\forall j : p_j \equiv 1 \!\!\!\!\pmod{4}$};	\\[0.5ex]
				3	\;,	&	\text{if $n$ is coprime to $5$ and $\exists j : p_j \equiv 3 \!\!\!\!\pmod{4}$};	\\[0.5ex]
				4	\;,	&	\text{if $n$ is a multiple of $5$}.
			\end{cases}
	\end{align*}
\end{theorem}
\begin{proof}
	We begin by characterizing $\udiam(\Gamma_n)$, where $n = p^m$ for $p \ge 5$ prime, using Lemma~\ref{lemma:countingSumSquares} throughout to characterize $S_n(r)$ for $r \in \Z_n$.
	\begin{itemize}
	\item
		If $p \equiv 1 \pmod{4}$ and $p > 5$, we have $S_{p^m}(r) > 0$ for all $r \in \Z_n$; then $\udiam(\Gamma_n) = 2$.

	\item
		If $p \equiv 3 \pmod{4}$ and $p > 5$, we have $S_{p^m}(r) = 0$ if and only if $r \in \Z_n$ is a zero divisor.
		In particular, $\udiam(\Gamma_n) \ge 3$.
		Conversely, as $\card{\Z_{p^m}\units} > p^{m-1}$, there exists $z \in Q_{p^m}\units$ such that $r - z$ is a unit; then there are quadratic units $x,y \in Q_{p^m}$ such that $r - z = x + y$, so that $\udiam(\Gamma_n) = 3$.
	
	\item
		If $p = 5$, we have $u \in Q_{5^m}$ if and only if $u \equiv \pm 1 \pmod{5}$; then $r$ can be expressed as a sum of $k$ quadratic units $r = u_1 + \cdots + u_k$ if and only if $r$ can be expressed modulo $5$ as a sum or difference of $k$ ones; that is, if $r \in \ens{-k , -k+2, \ldots, k-2, k} + 5 \Z_{5^m}$ (which exhausts $\Z_{5^m}$ for $k \ge 4$).
	\end{itemize}
	For $n$ not a prime power, we decompose $\Gamma_n \cong \Gamma_{\smash{p_1}^{\!m_1}} \ox \cdots \ox \Gamma_{\smash{p_t}^{\!m_t}}$; then a vertex $r = \sigma\inv(r_1, \ldots, r_t)$ is reachable by a walk of length $\ell$ in $\Gamma_n$ if and only if each $r_j \in V(\Gamma_{\smash{p_j}^{\!m_j}})$ are reachable by such a walk in their respective digraphs.
	Thus, the uniform diameter of the tensor product is the maximum of the uniform diameters of each factor.
\end{proof}
The uniform diameter $\Gamma_n$ happens also to provide an upper bound on distances between vertices in $G_n$, under the constraint that we may only traverse walks $w_0\, w_1 \, \ldots \, w_\ell$ where the ``type'' of each transition $w_j \to* w_{j+1}$ is fixed to be either a quadratic unit or the negation of a quadratic unit, independently for each $j$.
More precisely:
\begin{lemma}
	\label{lemma:oddSignWalks}
	Let $n = p_1^{m_1} \cdots p_t^{m_t}$ be relatively prime to $6$, and $\ell \ge \udiam(\Gamma_n)$.
	For any sequence $s_1, \ldots, s_\ell \in \ens{0,1}$, these exists a sequence of quadratic units $u_1, \ldots, u_\ell \in Q_n$ such that $r = (-1)^{s_1} u_1 + (-1)^{s_2} u_2 + \cdots + (-1)^{s_\ell} u_\ell$.
\end{lemma}
\begin{proof}
	We first show that there are solutions to $r = u_1 - u_2 \pm u_3 \pm \cdots \pm u_\ell$, where all but the first two signs may be arbitrary.
	We prove the result for $\ell = \udiam(\Gamma_n)$; one may extend to $\ell > \udiam(\Gamma_n)$ by induction.
	\begin{itemize}
	\item
		Suppose $n$ is coprime to $5$: then for any $r \in \Z_n$, we have $D_n(r) > 0$, so that there exist $u, u' \in Q_n$ such that $r = u - u'$.
		In the case that $n$ also has prime factors $p_j \equiv 3 \pmod{4}$, consider $s = r \mp u$ for any $u \in Q_n$: as there are solutions to $s = u - u'$ for $u,u' \in Q_n$, there are also solutions to $r = u - u' \pm u''$.

	\item
		Suppose $n = 5^{m_1} p_2^{m_2} \cdots p_t^{m_t}$.
		\begin{itemize}
		\item
			If $r \not\equiv \pm1 \pmod{5}$.
			Let $s \in \Z_n$ be such that $s \equiv 0 \pmod{5}$, and $s \not\equiv 0 \pmod{p_j}$ for any $p_j \ge 7$.
			Then $r - s \not\equiv \pm1 \pmod{5}$, so that $D_n(r) > 0$; by Lemma~\ref{lemma:oddPathsLength2}, there are then quadratic units $u_1, u_2 \in Q_n$ such that $r - s = u_1 - u_2$.
			We also have $S_n(s), S_n(-s), D_n(s) > 0$ by construction, which can be used to obtain decompositions $s = \pm u_3 \pm u_4$ for $u_3, u_4 \in Q_n$ depending on the choices of signs; we then have $r = u_1 - u_2 \pm u_3 \pm u_4$.

		\item
			If $r \equiv \pm 1 \pmod{5}$, consider $(r_1, \ldots, r_t) = \sigma(r)$.
			\begin{subequations}
			We select coefficients $u_j, u'_j \in Q_{\smash{p_j}^{\!m_j}}$ as follows.
			We set $u'_1 = -u_1 = r_1$, so that
			\begin{align}
				(r_1 - 2u_1) \,\equiv\, (r_2 + 2u'_2) \,\equiv\, (r_1 - u_1 + u'_1) \,=\, \pm3 \!\!\pmod{5}.
			\end{align}
			For each $p_j \ge 7$, we require $u_j \ne 2\inv r_j$ and $u'_j \notin \ens{-2\inv r_j,\, u_j - r_j}$, but may otherwise leave $u_j$ unconstrained; we then have
			\begin{align}
				(r_j - 2u_j),\, (r_j + 2u_j'),\, (r_j - u_j + u_j') \,&\ne\, 0	\qquad\text{for $p_j \ge 7$}.
			\end{align}
			Let $u = \sigma\inv(u_1, \ldots, u_t)$ and $u' = \sigma\inv(u'_1, \ldots, u'_t)$.
			By construction, we then have $D_n(r-2u), D_n(r+2u'), D_n(r-u+u') > 0$ by Lemma~\ref{lemma:oddPathsLength2}.
			\end{subequations}
			\begin{subequations}
			There then exist $u'', u''' \in Q_n$ such that
			\begin{align}
				r \,=&\, u''' - u'' + u + u	\,, \quad\text{or}
			  \\
				r \,=&\, u''' - u'' + u - u' \,=\, u''' - u'' - u' + u	\,,\quad\text{or}
			  \\
				r \,=&\, u''' - u'' - u' - u'\,,
			\end{align}
			selecting $u'', u'''$ according to the desired signs for the latter two terms.
			\end{subequations}
		\end{itemize}
	\end{itemize}
	Thus, there are solutions to $r = u_1 - u_2 \pm u_3 \pm \cdots \pm u_\ell$ for $u_j \in Q_n$ and $\ell = \udiam(\Gamma_n)$, for arbitrary choices of signs and $r \in \Z_n$.
	It follows that we may decompose $r = \pm u_1 \pm \cdots \pm u_\ell$ for arbitrary choices of sign, provided not all signs are the same.
	By considering walks in $\Gamma_n$ of length $\udiam(\Gamma_n)$ from $0$ to either $r$ or $-r$, we also have decompositions $r = u_1 + \cdots + u_\ell$ and $r = -u_1 - \cdots - u_\ell$ for suitable choices of $u_1, \ldots, u_\ell \in Q_n$.
\end{proof}

\noindent The principal motivation for Lemma~\ref{lemma:oddSignWalks} is to bound the diameters of graphs $G_n$ over tensor decompositions of the ring $\Z_n$:
\begin{lemma}
	\label{lemma:diameterFactors}
	Let $M, N \ge 1$ be relatively prime integers, and let $n = MN$.
	Then we have $\diam(G_n) \ge \max\ens{ \diam(G_N), \diam(G_M) }$.
	Furthermore, if $M$ is coprime to $6$, we have $\diam(G_n) \le \max\ens{ \diam(G_N), \udiam(\Gamma_M) + 1 }$ as well.
\end{lemma}
\begin{proof}
	Let $\rho : \Z_n \to \Z_N \oplus \Z_M$ be the natural isomorphism.
	Let $r \in \Z_n$ be arbitrary, and $(r', r'') = \rho(r)$.
	If $r = (-1)^{s_1} u_1 + \cdots + (-1)^{s_\ell} u_\ell$ for some $\ell > 0$ and $u_1, \ldots, u_\ell \in Q_n$, we also have
	\begin{subequations}
	\begin{align}
		 	r'
		\;=&\;\;
			(-1)^{s_1} \, u'_1 \,+\, \cdots \,+\, (-1)^{s_\ell} \, u'_\ell	\;,
		\\[1ex]
		 	r''
		\;=&\;\;
			(-1)^{s_1} \, u''_1 \,+\, \cdots \,+\, (-1)^{s_\ell} \, u''_\ell	\;,
	\end{align}
	\end{subequations}
	where $(u'_j, u''_j) = \rho(u_j)$.
	For $\ell = \diam(G_n)$, it follows that $\ell \ge \diam(G_M)$ and $\ell \ge \diam(G_N)$.
	
	Suppose further that $M$ is relatively prime to $6$: then $\udiam(\Gamma_M)$ is well-defined by Lemma~\ref{lemma:oddSignWalks}.
	For any $a \in \Z_N$, let $\ell > 0$ be the length of a walk in $G_N$ from $0$ to $\ell$: there are then $u_1, \ldots, u_\ell \in Q_{N}$ and $s_1, \ldots, s_\ell \in \ens{0,1}$ such that $a = (-1)^{s_1} u'_1 \,+ \cdots +\, (-1)^{s_\ell} u'_\ell$.
	If $\ell \ge \udiam(\Gamma_M)$, then for any $b \in \Z_M$, there also exist quadratic units $u''_1, \ldots, u''_\ell \in Q_M$ such that $b = (-1)^{s_1} u'''_1 \,+ \cdots +\, (-1)^{s_\ell} u'''_\ell$\,.
	We may always obtain such a walk of length $\ell \ge \udiam(\Gamma_M)$ in $G_N$ by taking the shortest walk from $0$ to $a$ in $G_N$, and repeatedly adding closed walks of length two to the end until we obtain a walk of length $\ell \ge \udiam(\Gamma_M)$.
	For such a walk, we then have
	\begin{align}
		\label{eqn:decomposeArbitrary}
			r
		\;=\;
			\rho\inv(a,b)
		\;=&\;
			\rho\inv\Bigg(\sum_{j = 1}^\ell (-1)^{s_j} u'_j \;\;,\;\; \sum_{j = 1}^\ell (-1)^{s_j} u''_j\Bigg)
		\notag\\=&\;
			\sum_{j = 1}^\ell
				(-1)^{s_j} \rho\inv(u'_j, u''_j)
		\;=\;
			\sum_{j = 1}^\ell
				(-1)^{s_j} u_j	\;,
	\end{align}
	for some choice of quadratic units $u_j = \rho\inv(u'_j, u''_j) \in Q_n$ and $s_j \in \ens{0,1}$.
	If $\diam(G_N) > \udiam(\Gamma_M)$, this construction yields path-lengths $\udiam(\Gamma_M) \le \ell  \le \diam(G_N)$; if instead $\udiam(\Gamma_M) \ge \diam(G_N)$, we obtain paths of length at most $\udiam(\Gamma_M) + 1$, which is saturated if there exist vertices $a \in V(G_N)$ whose distance $d_a$ from $0$ is such that $\udiam(\Gamma_M) - d_a$ is odd.
	In either case, we have $\diam(G_n) \le \max\ens{\diam(G_N), \udiam(\Gamma_M)+1}$.
\end{proof}

\subsection{Diameter of $G_n$ for $n$ even}

The notable differences between the cases of $n$ odd and $n$ even are due to the sparsity of the quadratic units in $\Z_{2^m}$ compared to that of powers of other primes, and also that the sum or difference of two units (quadratic or otherwise) is necessarily a zero divisor if $n$ is even.
This results in a significant increase of the maximum diameter in the case of $n$ even, compared to $n$ odd:
\begin{theorem}
	\label{thm:evenDiameter}
	Let $n > 0$ even.
	Let $\delta_3(n) = 1$ if $n$ has prime factors $p_j \equiv 3 \pmod{4}$ for $p_j > 3$, and $\delta_3(n) = 0$ otherwise. %; and $\gamma_5(n) = 1$ if $n$ is a multiple of $5$, and $\gamma_5(n) = 0$ otherwise.
	Then we have
	\begin{align}
		\label{eqn:evenDiameterFormula}
			\diam(G_n)
		=&\,
			\begin{cases}
				12,	&	\text{if $n$ is a multiple of $24$};																		\\[0.5ex]
				6	,	&	\text{if $n$ is an odd multiple of $12$};																	\\[0.5ex]
				5	,	&	\text{if $n$ is a multiple of $10$, but not of $12$};													\\[0.5ex]
% 				\!4	,\!										&	\text{if $n = 4m$ for any $m > 1$ coprime to $15$};				\\[0.5ex]
				4 ,					&	\text{if $n = 8K$ for $K > 0$ coprime to $15$};											\\[0.5ex]
				3 + \delta_3(n),	&	\text{if $n = 6K$ for $K > 0$ coprime to $10$};											\\[0.5ex]
				3 + \delta_3(n),	&	\text{if $n = 4K$ for $K > 1$ coprime to $30$};											\\[0.5ex]
				3	,					&	\text{if $n = 2K$ for $K > 1$ coprime to $30$};											\\[0.5ex]
				2	,					&	\text{if $n = 4$};																				\\[0.5ex]
				1	,					&	\text{if $n = 2$}.	
			\end{cases}
	\end{align}
	In particular, with Theorem~\ref{thm:oddDiameter}, we have $\diam(G_n) \le 12$ for all $n$, and $\diam(G_{p^m}) \le 4$ for any prime $p$ and $m \ge 0$.
\end{theorem}
\begin{proof}
	We use Lemma~\ref{lemma:diameterFactors} to reduce the task of characterizing $\diam(G_n)$ for $n$ even to a small collection of representative cases, by factoring $n = NM$ for suitable choices of coprime factors $N$ and $M$.
	\begin{itemize}
	\item 
		Suppose $n$ is a multiple of $12$.
		We may let $M$ be the largest factor of $n$ which is coprime to $12$, and $N = n/M$.
		\begin{itemize}
		\item
			If $N = 2^m 3^{m'}$ for $m \ge 3$, we then have $u \in Q_N$ if and only if $u \equiv 1 \pmod{8}$ and $u \equiv 1 \pmod{3}$, or equivalently if $u \equiv 1 \pmod{24}$.
			Then $T_N$ consists of those $q \in \Z_N$ such that $r \equiv \pm 1 \pmod{24}$.
			The distance of a vertex in $G_N$ from $0$ is then characterized by its residue modulo $24$, in which case we may show that $\diam(G_N) = 12$.

		\item
			Otherwise, $N = 4 \cdot 3^{m'}$, in which case $u \in Q_N$ if and only if $u \equiv 1 \pmod{4}$ and $u \equiv 1 \pmod{3}$, or equivalently if $u \equiv 1 \pmod{12}$.
			Then $T_N$ consists of those $q \in \Z_N$ such that $r \equiv \pm 1 \pmod{12}$; similarly as in the case above, we then have $\diam(G_N) = 6$.
	\end{itemize}
	Because $\diam(G_M), \udiam(\Gamma_M) \le 4$, we then have $\diam(G_n) = \diam(G_N)$ by Lemma~\ref{lemma:diameterFactors}.
	Thus $\diam(G_n) = 12$ if $N$ is a multiple of $24$; otherwise we have $\diam(G_n) = 6$.

	\item
		Suppose $n$ is a multiple of $10$, but not of $12$: specifically, $n$ is not a multiple of $60$.
		Let $M$ be the largest factor of $n$ which is coprime to $30$, and $N = n/M$.
		We may show that $T_N$ contains only residues which are equivalent to $\pm 1$ modulo $10$:
		\begin{itemize}
		\item
		  If $n$ is an odd multiple of $30$, we have $N = 2 \cdot 3^m \cdot 5^{m'}$.
		  Then $u \in Q_N$ if and only if $u$ is odd, $u \equiv 1 \pmod{3}$, and $u \equiv \pm 1 \pmod{5}$; equivalently, if $u \equiv 1 \pmod{30}$ or $u \equiv 19 \equiv -11 \pmod{30}$.

		\item
		  If $n$ is a multiple of $10$ but not of $30$, then without loss of generality $N = 2^{m_1} 5^{m_2}$.
		  We may show that $r \in Q_N$ if and only if both $r \equiv \pm 1 \pmod{5}$, and
		  \begin{align*}
				  r
			  \,\equiv\,
				  \begin{cases}
					  1	\!\!\!\!\pmod{2}	\!	&	\text{if $m_1 = 1$};		\\[1ex]
					  1	\!\!\!\!\pmod{4}	\!	&	\text{if $m_1 = 2$};		\\[1ex]
					  1	\!\!\!\!\pmod{8}	\!	&	\text{if $m_1 \ge 3$}.
				  \end{cases}
		  \end{align*}
		  In each case, we have $u \in Q_N$ if and only if $u \in \ens{1,9} \pmod{\bar N}$ for $\bar N = 10$, $\bar N = 20$, or $\bar N = 40$ respectively.
		\end{itemize}
		As $N$ is a multiple of $10$ in either case, vertices $r \in \Z_N$ such that $r \equiv 5 \pmod{10}$ can only be reached by a path from $0$ with length at least five, so that $\diam(G_N) \ge 5$.
		We may show that this bound is tight by showing that every even residue can be formed as a sum of four elements of $T_n$.
		\begin{subequations}
		Let $x \equiv_m y$ denote equivalence of two integers (or sets of integers) modulo $m$.
		Then, we may easily verify that
		\begin{align}
		\begin{split}
				\ens{\pm 1 \pm 1 \pm 1 \pm 1}
			\;\equiv_{\scriptscriptstyle 30}&\;
				\ens{26, 28, 0, 2, 4}	,
			\\
				\ens{\pm 1 \pm 1 \pm 1 \pm 11}
			\;\equiv_{\scriptscriptstyle 30}&\;
				\ens{8,10,12,14,16,18,20,22}	,
			\\
				-11 -11 -1 -1
			\;\equiv_{\scriptscriptstyle 30}&\;\;
				6 ,
			\\
				11 + 11 + 1 + 1
			\;\equiv_{\scriptscriptstyle 30}&\;\;
				24 ,
		\end{split}
		\end{align}
		which proves the claim for $n$ an odd multiple of $30$.
		For $n$ not a multiple of $30$, $T_N$ is the set of elements $q \in \Z_N$ such that $q = 5 \pm 4 \pmod{\bar N}$ or $q = -5 \pm 4 \pmod{\bar N}$.
		It then suffices to show that all residues modulo $40$ are exhausted by sums or differences of four such residues: we have
		\begin{align}
		\begin{split}
				\ens{\,\, (5\pm 4) + (5 \pm 4) + (5 \pm 4) + (5 \pm 4) \,\,}
			\;\equiv_{\scriptscriptstyle 40}&\;
				\ens{4, 12, 20, 28, 36}	,
			\\
				\ens{\,\, (5\pm 4) + (5 \pm 4) + (5 \pm 4) - (5 \pm 4) \,\,}
			\;\equiv_{\scriptscriptstyle 40}&\;
				\ens{34, \,2, 10, 18, 26}	,
			\\
				\ens{\,\, (5\pm 4) + (5 \pm 4) - (5 \pm 4) - (5 \pm 4) \,\,}
			\;\equiv_{\scriptscriptstyle 40}&\;
				\ens{24, 32, \,0, \,8, 16}	,
			\\
				\ens{\,\, (5\pm 4) - (5 \pm 4) - (5 \pm 4) - (5 \pm 4) \,\,}
			\;\equiv_{\scriptscriptstyle 40}&\;
				\ens{14, 22, 30, 38, \,6} .
		\end{split}
		\end{align}
		\end{subequations}
		As every odd residue modulo $N$ is adjacent to an even residue, it follows that every vertex in $G_N$ can be reached by a path of length at most five; then $\diam(G_N) = 5$.
		As $M$ is coprime to both $3$ and $5$, we have $\diam(G_M) = 2$ and $\udiam(\Gamma_M) \le 3$; thus $\diam(G_n) = 5$ by Lemma~\ref{lemma:diameterFactors}.

	\item
		Suppose that $n = 8 K$ for $K$ coprime to $15$.
		Let $M$ be the largest odd factor of $N$, and $N = n/M = 2^k$ for $k \ge 3$.
		By construction, $M$ is coprime to $6$, so that $\udiam(\Gamma_M) \le 3$.
		We have $u \in Q_N$ if and only if $u \equiv 1 \pmod{8}$: as every odd residue modulo $8$ can be expressed as a sum of three terms $\pm 1$, and every even residue modulo $8$ can be expressed as a sum of four terms $\pm 1$ (with $4$ requiring at least this many), it follows that $\diam(G_N) = 4$.
		By Lemma~\ref{lemma:diameterFactors}, it follows that $\diam(G_n) = 4$ as well.

	 \item
		In the remaining cases, we either have $n = 2K$ for $K$ coprime to $15$ and not a multiple of $4$, or $n = 6K$ for $K$ coprime to $10$.
		We trivially have $\diam(G_n) = \frac{n}{2}$ for $n \in \ens{2,4}$; otherwise, $n$ has odd zero divisors.
		As all walks of length one from $0$ in $G_n$ end at quadratic units, and all walks of length two from $0$ end at even elements of $\Z_n$, we require walks of length at least three from $0$ to reach odd zero divisors in $\Z_n$.
		Thus, $\diam(G_n) \ge 3$.

		Let $M$ be the largest factor of $n$ which is coprime to $30$.
		By construction, $M$ is coprime to $6$, so that $\udiam(\Gamma_M) = 2 + \delta_3(M) = 2 + \delta_3(n)$.
		\begin{itemize}
		\item
			If $n = 2K$ for $K$ coprime to $15$ and not a multiple of $4$, $M$ is simply the largest odd factor of $n$, in which case $n = 2^k$ for $k \in \ens{1,2}$.
			We then have $N \in \ens{2,4}$, so that $\diam(G_N) = \frac{1}{2}N \le 2$.

		\item
			If $n = 6K$ for $K$ coprime to $10$, we have $N = 2 \cdot 3^k$ for some $k \ge 1$.
			Then $u \in Q_N$ if and only if $u \equiv 1 \pmod{3}$ and is odd; that is, if $u \equiv 1 \pmod{6}$.
			In particular, $T_N$ contains only elements which are equivalent to $\pm 1 \pmod{6}$; from this we may easily show $\diam(G_N) = 3$.
		\end{itemize}
		In either case, it follows by Lemma~\ref{lemma:diameterFactors} that
		\begin{align}
				3	\,\le\, \diam(G_n) \,\le\, \udiam(\Gamma_M)	+ 1 \,=\, 3 + \delta_3(n).
		\end{align}
		If $\delta_3(n) = 0$, we then have $\diam(G_n) = 3$; we may then restrict our attention to the case $\delta_3(n) = 1$.
		
		Let $\rho: \Z_n \to \Z_M \oplus \Z_N$ be the natural isomorphism.
		As $M$ is coprime to $15$, we have $D_M(a) > 0$ for every $a \in \Z_M$ by Lemma~\ref{lemma:oddPathsLength2}: then we may express any $a \in \Z_M$ as a difference $a = u'_1 - u'_2$ for $u'_1, u'_2 \in Q_M$.
		\begin{itemize}
		\item
			If $N = 2$, any even residue $r$ may be expressed as $r = \rho\inv(a,0) = \rho\inv(u'_1,\,1) - \rho\inv(u'_2,\,1)$, which is a difference of the two quadratic units $u_j = \rho\inv(u'_j,\,1)$.
			Thus every even residue can be reached in $G_n$ by a path of length two from $0$.
			As every odd residue is adjacent to an even residue, we may reach any vertex by a path of length at most three; then $\diam(G_n) = 3$.

		\item
		  If $N = 4$ or $N$ is a multiple of $6$, we have $u \in Q_N$ if and only if $u \equiv 1 \pmod{\bar N}$, where $\bar N = 4$ if $N = 4$, and $\bar N = 6$ otherwise.
		  We may easily show that the only residues $r \in \Z_n$ which may be expressed as a difference of two quadratic units are those such that $r \equiv 0 \pmod{\bar N}$; and for any residue $a \equiv 0 \pmod{p_j}$, we have $S_M(a) = S_M(-a) = 0$ by Lemma~\ref{lemma:oddPathsLength2}.
		  Therefore, no residue $r = \rho\inv(a,\pm 2) \in \Z_n$ can be reached by a path of length two from $0$ in $G_n$.
		  As any sum of the form $\pm u_1 \pm u_2 \pm u_3$ will be odd for $u_1, u_2, u_3 \in Q_n$, such residues $r$ are in fact at a distance at least four from zero.
		  As $\diam(G_n) \le 3 + \delta_3(n) = 4$, it follows that $\diam(G_n) = 4 = 3 + \delta_3(n)$ in this case.
	  \end{itemize}
	\end{itemize}
	In each case, the diameters agree with the formula in \eqref{eqn:evenDiameterFormula}.
\end{proof}

\subsection{Perfectness}
\label{sec:perfect}

A graph $G$ is \emph{perfect}~\cite{AR01} if, for every induced subgraph $H \subset G$, the size $\omega(H)$ of the maximum clique in $H$ is equal to the chromatic number $\chi(H)$.
We may easily identify two classes of quadratic unitary graphs which are perfect:

\begin{lemma}
	\label{lemma:perfectSubcase}
	For $n$ even or $n = p^m$ for $p \equiv 3 \pmod{4}$ prime, $G_n$ is perfect.
\end{lemma}
\begin{proof}
	If $n$ is even, $G_n$ is bipartite, in which case $\omega(G_n) = \chi(G_n) = 2$ for any non-empty subgraph of $G_n$\,.
	Otherwise, suppose that $n = p^m$ for a prime $p \equiv 3 \pmod{4}$.
	Two vertices are adjacent in $G_{p^m}$ if and only if their residues modulo $p$ differ.
	For any $H \subset G_{p^m}$, a maximum-size clique in $H$ is then any set of vertices having different residues modulo $p$, where the number of different residues represented is chosen to be maximal; at the same time, any minimum colouring of $G_{p^m}$ maps each residue class modulo $p$ to a common colour, with different residue classes having different colours.
	Thus, $\omega(H) = \chi(H)$ for all such $H$, so that $G_{p^m}$ is perfect for $p \equiv 3 \pmod{4}$.
\end{proof}

A perfect graph $G$ contains no \emph{odd holes} (induced cycles of length $2k+1$ for $k > 1$, which have no cliques larger than two but are not bipartite).
Chudnovsky, Robertson, Seymour, and Thomas~\cite{CRST02} characterized perfect graphs in terms of odd holes, proving a conjecture of Berge~\cite{Berge}:

\medskip
\begin{named theorem}[Strong Perfect Graph Theorem]
  A graph $G$ is perfect if and only if neither $G$ nor its complement $\bar G$ contain odd holes.
\end{named theorem}
\medskip

There exist imperfect quadratic unitary Cayley graphs $G_n$\,.
In particular, quadratic unitary Cayley graphs $G_p$ for $p \equiv 1 \pmod{4}$ are also circulant Paley graphs, which Maistrelli and Penman~\cite{MP06} show are imperfect by exploiting the fact that they are self-complementarity.\footnote{%
  This is in fact the simplest case of a more comprehensive theorem, which shows that the only perfect Paley graph is that on nine vertices.
}
The Strong Perfect Graph Theorem implies that these graphs contain odd holes (again by self-complementarity) .
We may consider how to ``lift'' odd holes in such graphs $G_p$ to obtain odd holes in graphs $G_n$ having such factors $p$; and we may use a similar strategy for $G_n$ having distinct prime factors $p_1, p_2 \equiv 3 \pmod{4}$.
That is:
\begin{theorem}
	\label{thm:perfectCharn}
	$G_n$ is perfect if and only if $n$ is even, or $n = p^m$ for $p \equiv 3 \pmod{4}$ prime.
\end{theorem}
\begin{proof}
	Building on Lemma~\ref{lemma:perfectSubcase}, we demonstrate that $G_n$ has an odd hole if $n$ is odd and is not a power of a prime $p \equiv 3 \pmod{4}$.
	We proceed by constructing a simpler graph $G_\nu$, for $\nu$ a factor of $n$, which has an odd hole.
	Let $n = p_1^{m_1} p_2^{m_2} \cdots p_t^{m_t}$\,:
	\begin{enumerate}
	\item
	  Suppose that $n$ has a prime factor which is equivalent to $1 \pmod{4}$.
	  Without loss of generality, $p_1 \equiv 1 \pmod{4}$; we then let $\nu = p_1$.
	  By~\cite{MP06}, $G_\nu$ is then not perfect, and so $G_\nu$ or its complement has an odd hole.
	  Note that $G_\nu$ is self-complementary, as multiplication of any pair of adjacent vertices by a non-quadratic unit $r$ yields a non-adjacent pair, and vice-versa.
	  Thus, $G_\nu$ contains an odd hole $u_0 ~ u_1 ~ u_2 ~ \cdots ~ u_{\ell-1} ~ u_0$ of some odd length $\ell$.

	\item
	  Suppose instead that $n$ has two distinct prime factors equivalent to $3 \pmod{4}$.
	  Without loss of generality, $p_1,p_2 \equiv 3 \pmod{4}$; we then define $\nu = p_1 p_2$\,, and let $\rho: \Z_\nu \to \Z_{p_1} \oplus \Z_{p_2}$ be the natural isomorphism.
	  We demonstrate the existence of an induced five-cycle $u_0 ~ u_1 ~ u_2 ~ u_3 ~ u_4 ~ u_0$ in $G_\nu$ by constructing appropriate structures in the digraphs $\Gamma_{p_1} \ox \Gamma_{p_2} = \rho(\Gamma_\nu)$\,.
	  Let
	  \begin{align}
		u_0 \;=\; 0 \;=&\; \rho\inv(0,0)	\,,	&	  
		u_1 \;=\; 1 \;=&\; \rho\inv(1,1)	\,.
	  \end{align}
	  We proceed by cases:
	  \begin{itemize}
	  \item
		Suppose $2$ is not a quadratic residue modulo either $p_1$ or $p_2$\,.
		Without loss of generality, we may suppose $3 \le p_1 < p_2$; in particular, $p_2 = 11$ or $p_2 \ge 19$.
		By~\cite{B1925}, there are then consecutive triples $q, q+1, q+2 \in Q_{p_2}$.
		Let $q \in Q_{p_2}$ be a minimal such residue: as $2 \notin Q_{p_2}$\,, it follows that $q \ne 1$, so that $q-1 \in -Q_{p_2}$.
		In particular, there is an arc from $q+1$ to $2$.
		Define the vertices
	  	\begin{align}
			\rho(u_2) \;=&\; (2,q+1)	\,,	&
			\rho(u_3) \;=&\; (0,2)	\,,	&
			\rho(u_4) \;=&\; (2,-q)	\,.
		\end{align}
		The vertices are distinct; in particular, $u_2$ differs from $u_4$, because our choice that $q+1 \in Q_{p_2}$ implies that $q+1 \ne -q$.
		By agreement of arc-directions $x_j \arc x_k \in E(\Gamma_{p_1})$ and $y_j \arc y_k \in E(\Gamma_{p_2})$ for pairs of coefficients $(x_k, y_k) = \rho(u_k)$, we have the arc-relations
		\begin{align}
			u_0 \arc u_1 \arc u_2 \arc u_3 \cra u_4 \cra u_0
		\end{align}
		in $\Gamma_\nu$\,.
		There are no other arc relations among the vertices $u_k$\,.
		In particular, $u_0$ and $u_3$ are non-adjacent as they have the same residue modulo $p_1$\,, and similarly for $u_2$ and $u_4$\,; and there are no arcs between the other pairs of vertices $u_j$ and $u_k$ because we have $x_j - x_k \in \pm Q_{p_1}$ if and only if $y_j - y_k \in \mp Q_{p_2}$\,.

	\item
	  Suppose that $2$ is a quadratic residue modulo only one of $p_1$ or $p_2$\,.
	  Without loss of generality, we may suppose $2 \in Q_{p_2}$\,, in which case $p_2 \ge 7$\,.
	  Define the vertices
	  \begin{align}
			\rho(u_2) \;=&\; (2,2)	\,,	&
			\rho(u_3) \;=&\; (0,3)	\,,	&
			\rho(u_4) \;=&\; (1,4)	\,.
	  \end{align}
	  Clearly the vertices $u_j$ are distinct.
	  By agreement of arc-directions $x_j \arc x_k \in E(\Gamma_{p_1})$ and $y_j \arc y_k \in E(\Gamma_{p_2})$ for pairs of coefficients $(x_k, y_k) = \rho(u_k)$, we have the arc-relations
	  \begin{align}
		u_0 \arc u_1 \arc u_2 \arc u_3 \arc u_4 \cra u_0\
	  \end{align}
	  in $\Gamma_\nu$\,.
	  There are no other arc relations among the vertices $u_k$\,.
	  In particular, $u_0$ and $u_3$ are non-adjacent as they have the same residue modulo $p_1$\,, and similarly for $u_1$ and $u_4$\,; and there are no arcs between the other pairs of vertices $u_j$ and $u_k$ because we have $x_j - x_k \in \pm Q_{p_1}$ if and only if $y_j - y_k \in \mp Q_{p_2}$\,.
	
  \item
	  Suppose that $2$ is a quadratic residue modulo both $p_1$ and $p_2$\,.
	  In particular, $p_1, p_2 \ge 7$, so that there are pairs of consecutive quadratic residues modulo $p_1$ and $p_2$ (see Lemma~\ref{lemma:consecutiveResidues}).
	  Let $q_1-1, q_1 \in Q_{p_1}$ and $q_2-1, q_2 \in Q_{p_2}$ be the largest such pairs in each case; in particular it follows that $q_2 \ne -1$, so that $-q_2 -1 \in Q_{p_2}$\,.
	  Define the vertices
	  \begin{align}
		  \rho(u_2) \;=&\; (q_1,-q_2) \,,	&
		  \rho(u_3) \;=&\; (0,1) \,,	&
		  \rho(u_4) \;=&\; (1,q_2) \,.
	  \end{align}
	  The vertices are again distinct; and by arc agreement $x_j \arc x_k \in E(\Gamma_{p_1})$ and $y_j \arc y_k \in E(\Gamma_{p_2})$ for pairs of coefficients $(x_k, y_k) = \rho(u_k)$, we have the arc-relations
	  \begin{align}
		  u_0 \arc u_1 \arc u_2 \cra u_3 \arc u_4 \cra u_0
	  \end{align}
	  in $\Gamma_\nu$\,.
	  There are no other arc relations among the vertices $u_k$\,.
	  In particular, $u_0$ and $u_3$ are non-adjacent as they have the same residue modulo $p_1$\,, and similarly for $u_1$ and $u_4$\,; the vertices $u_1$ and $u_3$ have the same residue modulo $p_2$\,.
	  By construction, we have $2q_2 \in Q_{p_2}$ and $1 - q_1 \in -Q_{p_1}$\,, so there is no arc between $u_2$ and $u_4$\,; and similarly there is no arc between $u_0$ and $u_2$\,.
	  \end{itemize}
	\end{enumerate}
	In each case, $G_\nu$ contains an induced cycle $u_0 ~ u_1 ~ u_2 ~ \cdots ~ u_{\ell-1} ~ u_0$ for some odd $\ell \ge 5$.
	Let
	\begin{align}
		  \mu
	  \;=\;
		  \begin{cases}
			  p_1^{m_1}	,			&	\text{if $\nu = p_1$},		\\[1ex]
			p_1^{m_1} p_2^{m_2}	,	&	\text{if $\nu = p_1 p_2$}:
		  \end{cases}
	\end{align}
	we may then obtain a similar odd hole in $G_\mu$ by identifying each $u_j$ with a corresponding $x_j \in \ens{0, \ldots, \nu-1} \subset \Z_\mu$.
	Then $x_j - x_k \in Q_\mu$ if and only if $u_j - u_k \in Q_\nu$ for any $1 \le j,k \le \ell$, which implies that $x_0 ~ x_1 ~ \cdots ~ x_{\ell-1} ~ x_0$ is a cycle without chords in $G_\mu$.

	Let $N$ be the largest factor of $n$ which is coprime to $\mu$ (\ie\ $N = n / \mu$), and let $\tau: \Z_n \to \Z_\mu \oplus \Z_N$ be the natural isomorphism.
	\begin{itemize}
	\item
		Suppose $\nu = p_1 \equiv 1 \pmod{4}$.
		If $n$ is a multiple of $5$, we may suppose that $p_1 = 5$ without loss of generality; then $N$ is coprime to $5$.
		By Lemma~\ref{lemma:oddPathsLength2}, we then have $S_N(r) > 0$ for any $-r \in Q_N$: then $G_N$ contains a closed walk $0 \; u \; r \; 0$ for some $u \in Q_N$.
		By concatenating this walk with $\frac{1}{2}(\ell - 3)$ copies of a closed walk of length two at $0$, we obtain a closed walk $y_0 ~ y_1 ~ \cdots ~ y_{\ell-1} \, y_0$ in $G_N$\,.
		We may then construct a walk
		\begin{align}
		  \label{eqn:constructedOddHole}
		  C = (x_0, y_0) ~ (x_1, y_1) ~ \cdots ~ (x_{\ell-1}, y_{\ell-1}) ~ (x_0, y_0)
		\end{align}
		in $G_\mu \ox G_N$: because $x_0 ~ x_1 ~ \cdots ~ x_{\ell-1} ~ x_0$ is an induced cycle, so is $C$.

	\item
	  Otherwise, suppose $\mu = p_1^{m_1} p_2^{m_2}$.
	  If $n$ is a multiple of $3$, we may suppose that $p_1 = 3$ without loss of generality; then $N$ is coprime to $3$.
	  By Theorem~\ref{thm:oddUniformDiameter}, we then have $\ell > \udiam(\Gamma_N)$, in which case by Lemma~\ref{lemma:oddSignWalks} we may construct a closed walk $y_0 ~ y_1 ~ \cdots ~ y_{\ell-1} ~ y_0$ in $G_N$ by setting $y_0 = 0$, and letting $y_{j+1} - y_j \in \pm Q_N$ whenever $x_{j+1} - x_j \in \pm Q_\mu$ (\ie~with agreement in the signs), and similarly for $y_0 - y_{\ell-1}$.
	  Define the walk $C$ in $\Z_\mu \oplus \Z_N$ as given in \eqref{eqn:constructedOddHole}: we then have $(x_{j+1}, y_{j+1}) - (x_j, y_j) \in \pm (Q_\mu \oplus Q_\nu)$ for each $j$, and similarly $(x_0, y_0) - (x_{\ell-1}, y_{\ell-1}) \in \pm (Q_\mu \oplus Q_N)$.
  \end{itemize}
  In either case, if we define vertices $v_j = \tau\inv(x_j, y_j) \in V(G_n)$, the walk $v_0 ~ v_1 ~ \cdots ~ v_{\ell-1} \, v_0$ is an induced cycle of odd length in $G_n$.
  Thus $G_n$ is not perfect, unless $n$ is even or a power of a prime $p \equiv 3 \pmod{4}$.
\end{proof}

\section{Decomposing symplectic operators mod $n$}

Our final result is a bound on the complexity of decompositions of symplectic operastors modulo $n$, which follows from the bound on the diameter of $G_n$.
We may define the \emph{symplectic form} (modulo $n$) as the $2m \x 2m$ matrix
\begin{align}
	\sigma_{2m}
  \;=\;
	\left[\,\begin{matrix}
	        	0_m	& \!-I_m	\\[0.5ex]
				\;I_m\; & 0_m
	        \end{matrix}\,\right]
  \,;
\end{align}
the \emph{symplectic group modulo $n$} $\Sp_{2m}(\Z_n)$ is the set of $2m \x 2m$ linear operators $S$ (\emph{symplectic operators}) with coefficients in $\Z_n$ such that $S\trans \sigma_{2m} S = \sigma_{2m}$.

\begin{convention}
  For operators $S \in \Sp_{2m}(\Z_n)$ for a fixed $m$, we will adopt the convention of indexing the rows and columns by integers modulo $2m$, starting with $1$.
  Thus, for a row $k \in \ens{m+1, \ldots, 2m}$ in the ``bottom'' half of a matrix $S$, the row $k+m \in \ens{1, \ldots, m}$ will be in the ``top'' half, and vice-versa.
\end{convention}

Symplectic operators are clearly invertible operations, and therefore may be reduced to $I_{2m}$ by Gaussian elimination.
We also consider a variant procedure, in which row-operations are constrained to themselves be symplectic.
For the operator definitions below, actions of operators are defined via the action of left-multiplication on a square matrix over $\Z_n$.
\begin{definition}
	For row-indices $j,k \in \ens{1, \ldots, 2m}$, a \emph{symplectic row operation acting on $\Sp_{2m}(\Z_n)$} is one of the operators $M\supp{\alpha}_j$, $E_{j,k}$\,, $C_{j,k}$\,, or $C_{j,k}^{-1}$ defined as follows:
\begin{itemize}
\item
  For any $\alpha \in \Z_n\units$, let $\mu_j\supp{\alpha} \in \GL_{2n}(\Z_n)$ be the linear operator which multiplies the $j\textsuperscript{th}$ row of its operand by $\alpha$.
  Then, we define $M_j\supp{\alpha} \,=\, \mu_j\supp{\alpha}\, \mu_{j+m}\supp{\smash{\alpha\inv}}$.

\item
  Let $\epsilon_{j,k} \in \GL_{2m}(\Z_n)$ be the linear operator which exchanges rows $j$ and $k$ of its operand.
  Then we define
  \begin{align}
		E_{j,k}
	  \;=\;
		\begin{cases}
			\epsilon_{j,k}	\, \mu_k\supp{-1} \,,	&	\text{if $j - k \equiv m \pmod{2m}$};	\\[1ex]
			\epsilon_{j,k}	\, \epsilon_{j+m,k+m}	&	\text{otherwise},
		\end{cases}
  \end{align}
  for $\mu_j\supp{-1}$ as defined above.

\item
  Let $\chi_{j,k} \in \GL_{2m}(\Z_n)$ be the linear operator which adds row $j$ of its operand to row $k$.
  Then we define
  \begin{align}
		C_{j,k}
	  \;=\;
		\begin{cases}
			\chi_{j,k}	\,,								&	\text{if $j - k \equiv m \pmod{2m}$};	\\[1ex]
			\chi_{j,k}	\, \chi_{k+m,j+m}^{-1}	\,,		&	\text{if $1 \le j,k \le m$ or $m+1 \le j,k \le 2m$}; \\[1ex]
			\chi_{j,k}	\, \chi_{k+m,j+m}		 \,,	&	\text{otherwise}.
		\end{cases}
  \end{align}
\end{itemize}
\end{definition}
These operations are defined so as to be symplectic themselves; we wish to demonstrate an upper bound to the number of such symplectic row operations required to transform an arbitrary symplectic operator to the identity.

Hostens \emph{et al.}~\cite{HDM05} provide a decomposition of symplectic operators into $O(m^2 \log(n))$ symplectic row operations, in an application to the the decomposition of an important family of unitary operators for quantum computation (specifically, the Clifford group over qudits of dimension $n$).
We refine this decomposition to obtain an upper bound to $O(m^2)$, giving an upper bound which is independent of the modulus $n$.

\subsection{Reduction to greatest common divisors modulo $n$}

We first describe the decomposition of~\cite{HDM05} in detail.
The main concept is to reduce $S \in \Sp_{2m}(\Z_n)$ to another operator $S'$ which acts trivially on, \eg, the standard basis vectors $\unit_m, \unit_{2m}$.
This reduces the problem to decomposing an operator $\tilde S \in \Sp_{2m-2}(\Z_n)$,
\begin{align}
  \label{eqn:symplecticGroupReduction}
		\tilde S
	\;=&\;
		\left[\,\begin{matrix}
					A'_{11}	&	A'_{12}	\\[0.5ex]
					A'_{21}	&	A'_{22}
		\end{matrix}\,\right]
  &
  \text{for}\;\;
		S'
	\,=&\;
		\left[\;\begin{array}{c@{\,}c@{\,}c|c|c@{\,}c@{\,}c|c}
			&				&		&		0		&		&				&		&		0		\\[-0.8ex]
			&	A'_{11}	&		&	\vdots	&		&	A'_{12}	&		&	\vdots	\\
			&				&		&		0		&		&				&		&		0		\\
		\hline
		0	&	\cdots	&	0	&		\;1\;		&	0	&	\cdots	&	0	&		0		\\
		\hline
			&				&		&		0		&		&				&		&		0		\\[-0.8ex]
			&	A'_{21}	&		&	\vdots	&		&	A'_{22}	&		&	\vdots	\\
			&				&		&		0		&		&				&		&		0		\\
		\hline
		0	&	\cdots	&	0	&		0		&	0	&	\cdots	&	0	&		\;1\;
	\end{array}\;\right].
\end{align}
Embedding the matrix groups $\Sp_2(\Z_n) \subset \cdots \subset \Sp_{2m-2}(\Z_n) \subset \Sp_{2m}(\Z_n)$ in the manner described above, one may recursively apply this process to obtain a sequence of symplectic row operations which multiply to transform $S$ to $I_{2m}$.
As the inverse of each symplectic row operation is also a symplectic row operation, this yields a decomposition of $S$.

The reduction from $S$ to $S'$ as above is performed by a hybrid of Gaussian elimination and Euclid's algorithm for computing greatest common divisors.
\begin{subequations}
We illustrate this on a $2m \x 2$ matrix $[\; \vec v \; \vec w \;]$, for a pair of column vectors $\vec v = [\; v_1 \;\; v_2 \;\; \cdots \;\; v_{2m} \;]\trans$ and $\vec w = [\; w_1 \;\; w_2 \;\; \cdots \;\; w_{2m} \;]\trans$ subject to the constraint $\vec w\trans \sigma_{2m} \vec v = 1$.
By performing suitable symplectic row-additions, we may simulate the Euclidean algorithm in the second column, for each pair of rows $(j,\, j+m)$ for $j \in \ens{1, \ldots, m}$, to obtain
\begin{align}
  \label{eqn:localGCDcompute}
		\left[\;\begin{matrix}
			v_1	& w_1	\\[0.2ex]	v_2	& w_2	\\[0.2ex]	\vdots & \vdots	\\[0.2ex]	v_m	& w_m
		\\[0.5ex]
		\hline
		\\[-2ex]
		  v_{m+1}	& w_{m+1}	\\[0.2ex]	v_{m+2}	& w_{m+2}	\\[0.2ex]	\vdots	& \vdots \\[0.2ex]	v_{2m}	& w_{2m}
	  \end{matrix}\;\right]
  \,\mapsto\,
		\left[\;\begin{matrix}
			\tilde v_1	& 0	\\[0.2ex]	\tilde v_2 & 0	\\[0.2ex]	\vdots & \vdots	\\[0.2ex]	\tilde v_m & 0
		\\[0.5ex]
		\hline
		\\[-2ex]
		  \tilde v_{m+1}	& \gcd(w_1,w_{m+1},n)	\\[0.2ex]	\tilde v_{m+2} & \gcd(w_2,w_{m+2},n)	\\[0.2ex]	\vdots	& \vdots \\[0.2ex]	\tilde v_{2m}	& \gcd(w_m,w_{2m},n)
	  \end{matrix}\;\right]
  \,=:\,
		\left[\;\begin{matrix}
			\tilde v_1	& 0	\\[0.2ex]	\tilde v_2 & 0	\\[0.2ex]	\vdots & \vdots	\\[0.2ex]	\tilde v_m & 0
		\\[0.5ex]
		\hline
		\\[-2ex]
		  \tilde v_{m+1}	& \gamma_1	\\[0.2ex]	\tilde v_{m+2} & \gamma_2	\\[0.2ex]	\vdots	& \vdots \\[0.2ex]	\tilde v_{2m}	& \gamma_m
	  \end{matrix}\;\right],
\end{align}
computing ``greatest common divisors'' (modulo $n$) in the lower block in the second column, and using these to clear the upper block.
We then perform further row-additions to compute further greatest common divisors in the second column, in pairs of rows $(j, j+1)$ for $j \in \ens{m+1, \ldots, 2m-1}$, in to perform the following transformation of the the second column:
\begin{align}
  \label{eqn:staircaseGCDcompute}
		\left[\;\begin{matrix}
			0					\\	0				\\	\vdots	\\	0	\\
		\hline
			\gamma_1	\\	\gamma_2		\\	\vdots	\\	\gamma_m
	  \end{matrix}\;\right]
  \,\mapsto\,
		\left[\;\begin{matrix}
			0					\\	0	\\	\vdots	\\	0	\\
		\hline
			0	\\	\gcd(\gamma_1,\gamma_2)	\\	\vdots	\\	\gamma_m
	  \end{matrix}\;\right]
  \,\mapsto\,
  \cdots
  \,\mapsto\,
 		\left[\;\begin{matrix}
 			0	\\	0	\\	\vdots	\\	0	\\
 		\hline
 			0	\\	0	\\	\vdots	\\	\gcd(\gamma_1, \gamma_2, \ldots, \gamma_{2m})
 	  \end{matrix}\;\right].
\end{align}
\end{subequations}
Note that as $\vec w\trans \sigma_{2m} \vec v = 1$, there is an integer combination of the coefficients of $\vec w$ which is equivalent to $1$ modulo $n$; then
\begin{align}
  \gcd(\gamma_1, \ldots, \gamma_m) \;=\; \gcd(w_1, \ldots, w_{2m}, n) \;=\; 1\,.
\end{align}
The above row-transformations then transform the two-column matrix $[\; \vec v \;\, \vec w\;]$ as follows:
\begin{align}
		\left[\;\begin{matrix}
			v_1	& w_1	\\[0.2ex]	v_2	& w_2	\\[0.2ex]	\vdots & \vdots	\\[0.2ex]	v_m	& w_m
		\\[0.5ex]
		\hline
		\\[-2ex]
		  v_{m+1}	& w_{m+1}	\\[0.2ex]	v_{m+2}	& w_{m+2}	\\[0.2ex]	\vdots	& \vdots \\[0.2ex]	v_{2m}	& w_{2m}
	  \end{matrix}\;\right]
  \,\mapsto*\,
 		\left[\;\begin{matrix}
 			v'_1 & 0	\\[0.2ex]	v'_2 & 0	\\[0.2ex]	\vdots & \vdots	\\[0.2ex]	1 & \;\;0\;\;
		\\[0.5ex]
 		\hline
		\\[-2ex]
 			v'_{m+1} & 0	\\	v'_{m+2} & 0	\\	\vdots & \vdots	\\	v'_{2m} & 1
 	  \end{matrix}\;\right]
	\,=:\,
	\Big[\; \vec v' \;\; \unit_{2m} \;\Big],
\end{align}
where $v'_m = 1$ follows from $\unit_{2m}\trans \sigma_{2m} \vec v' = \vec w\trans \sigma_{2m} \vec v = 1$.
We may repeat the sequence of transformations to compute greatest common divisors in the first column, for each pair of rows $(j,j+m)$ for $j \in \ens{1, \ldots, m}$, and subsequently in row-pairs $(j,j+1)$ in the upper block:
\begin{align}
  \label{eqn:symplecticReduceFirstColumn}
 		\left[\;\begin{matrix}
 			v'_1 & 0	\\[0.2ex]	v'_2 & 0	\\[0.2ex]	\vdots & \vdots	\\[0.2ex]	1 & \;0\;
		\\[0.5ex]
 		\hline
		\\[-2ex]
 			v'_{m+1} & 0	\\	v'_{m+2} & 0	\\	\vdots & \vdots	\\	v'_{2m} & 1
 	  \end{matrix}\;\right]
	\,\mapsto*\,
 		\left[\;\begin{matrix}
 			\varphi_1 & 0	\\[0.2ex]	\varphi_2 & 0	\\[0.2ex]	\vdots & \vdots	\\[0.2ex]	1 & \;0\;
		\\[0.5ex]
 		\hline
		\\[-2ex]
 			0 & 0	\\	0 & 0	\\	\vdots & \vdots	\\	0 & 1
 	  \end{matrix}\;\right]
	\,\mapsto*\,
 		\left[\;\begin{matrix}
 			0 & 0	\\[0.2ex]	\gcd(\varphi_1,\varphi_2) & 0	\\[0.2ex]	\vdots & \vdots	\\[0.2ex]	1 & \:\;0\:\;
		\\[0.5ex]
 		\hline
		\\[-2ex]
 			0 & 0	\\	0 & 0	\\	\vdots & \vdots	\\	0 & 1
 	  \end{matrix}\;\right]
	\,\mapsto*\,
 		\left[\;\begin{matrix}
 			0 & 0	\\[0.2ex]	0 & 0	\\[0.2ex]	\vdots & \vdots	\\[0.2ex]	1 & \;0\;
		\\[0.5ex]
 		\hline
		\\[-2ex]
 			0 & 0	\\	0 & 0	\\	\vdots & \vdots	\\	\;0\; & 1
 	  \end{matrix}\;\right].
\end{align}
The reduction of~\cite{HDM05} applies this procedure for $\vec v = S\unit_m$, $\vec w = S\unit_{2m}$.
Applying these transformations to $S$ yields a matrix $S'$ as illustrated in \eqref{eqn:symplecticGroupReduction}, as the other columns $S' \unit_k$ for $k \notin \ens{m,2m}$ must satisfy
\begin{align}
	\unit_k\trans S' \sigma_{2m} \unit_m \;=\; \unit_k S \sigma_{2m} S \unit_m \;=\; 0,
\end{align}
and similarly $\unit_k\trans S' \sigma_{2m} \unit_{2m} = 0$.
The complexity of a single iteration of this reduction is $O(m \log(n))$, which arises from the cost of repeating Euclid's algorithm (expressed in fixed-width integer addition steps) $O(m)$ times to reduce the $m\th$ and $2m\th$ columns to $\unit_m$ and $\unit_{2m}$ respectively.
Iterated $m$ times over all column-pairs $\unit_j, \unit_{j+m}$, we obtain the upper bound of $O(m^2 \log(n))$ reported by~\cite{HDM05}.

\subsection{Improved upper bounds via the diameter of $G_n$}
\bgroup
\def\inv{^{\text{--}1}}

The complexity of the above decomposition may be reduced to $O(m^2)$, by substituting an explicit simulation of Euclid's algorithm via symplectic row transformations with a product of constant size.
This is possible by using short paths in the graphs $G_n$ to reduce the number of addition steps in order to obtain coefficients $\gamma_j$ and $\varphi_j$ (or coefficients equivalent to them, up to a multiplicative unit) using a constant number of row-operations.

The primary obstacle to reducing the complexity of a single iteration of the reduction of~\cite{HDM05} is the computation of greatest common divisors in row-pairs $(j,j+m)$, arising from constraints on obtaining ``derived'' row-additions on these row-pairs.
\begin{subequations}
\label{eqn:derivedRowAdditions}
The iterated operator $C_{j,k}^{\,\alpha}$ (for $\alpha \in \ens{0, 1, \ldots, n-1}$, which we identify with $\alpha \in \Z_n$) can be easily obtained in constant depth for $j \not\equiv k + m \pmod{2m}$ and $\alpha \in \Z_n\units$, by the equality
\begin{align}
	C_{j,k}^{\,\alpha}
  \:=\;
	M\supp{\smash{\alpha\inv}}_j	\,
	C_{j,k}	\;
	M\supp{\alpha}_j	\;,
\end{align}
which one may verify by the action on standard basis vectors.
However, $C_{j,j+m}^{\,\alpha}$ cannot be decomposed in this manner: the closest we may come is in the case where $\alpha = u^2$ for some $u \in \Z_n\units$, in which case we have
\begin{align}
	C_{j,k}^{\,\alpha}
  \:=\;
	C_{j,k}^{\,u^2}
  \:=\;
	M\supp{\smash{u\inv}}_j	\,
	C_{j,j+m}	\;
	M\supp{u}_j	\;.
\end{align}
\end{subequations}
We may apply the result of Theorem~\ref{thm:evenDiameter} as follows:
\begin{lemma}
	For distinct row-indices $j,k \in \ens{1, \ldots, 2m}$ and for any $\alpha \in \Z_n$, there exists a sequence of units $a_1, \ldots, a_\ell \in \Z_n\units$ and signs $s_1, \ldots, s_\ell \in \ens{-1, +1}$ for some $\ell \le 12$, such that
	\begin{align}
		C_{j,k}^{\,\alpha}
% 	  \;=&\:\:
% 		C^{s_1 u_1^2}_{j,j+m}	\;
% 		C^{s_2 u_2^2}_{j,j+m}	\;
% 		\cdots	\;
% 		C^{s_\ell u_\ell^2}_{j,j+m}
% 	  \notag\\
\;=&\:\:
		M_j^{a_1\inv}  C^{\,s_1}_{j,k}	\,
		M_j^{a_2\inv a_1} \, C^{\,s_2}_{j,k} \,
		M_j^{a_3\inv a_2} \, \cdots \,
		M_j^{a_\ell\inv a_{\ell-1}} C^{\,s_\ell}_{j,k} \, M_j^{a_\ell}	\;.
	\end{align}
\end{lemma}
\begin{proof}
	It suffices to note that as $\diam(G_n) \le 12$, there exists such a sequence of signs and quadratic units $u_1, \ldots, u_\ell \in Q_n$ such that $\alpha = s_1 u_1 + s_2 u_2 + \cdots + s_\ell u_\ell$.
	We may then take either $a_j = u_j$ (in the case that $k \ne j + m$) or a unit $a_j$ such that $u_j = a_j^2$ (in the case that $k = j + m$), and apply the decompositions of \eqref{eqn:derivedRowAdditions} to obtain the desired decomposition.\footnote{%
	  For $k \ne j + m$, we may in fact obtain the further bound of $\ell \le 3$, as the diameter of the unitary Cayley graph $X_n = \Cay(\Z_n, \Z_n\units)$ is at most three~\cite{KS07}.}
\end{proof}
We may apply this to reduce the complexity of decomposing symplectic operators as follows.
We use the following additional Lemma, whose proof is deferred to the appendix:
\begin{lemma}
	\label{lemma:Bezout+n}
 	Let $\gamma = \gcd(x, y, n)$: then there exist $a,b,c \in \Z$ such that $ax + by + cn = \gamma$ and where both $a$ and $b$ are relatively prime to $n$.
\end{lemma}
For a vector $\vec x = [\; x_1 \;\; x_2 \;\; \cdots \;\; x_{2m} \;]\trans$, let $\gamma_j = \gcd(x_j, x_{j+m}, n)$ for each $j \in \ens{1, \ldots, m}$.
Let $a_j$ be coefficients such that $a_{j+m} x_{j+m} \,+\, a_j x_j \,\equiv\, \gamma_j \pmod{n}$ as guaranteed by Lemma~\ref{lemma:Bezout+n}, and define $r_j = a_{j+m}^{-1} a_j$: we then have
\begin{align}
  C^{\,r_1}_{1,m+1}
  \:\! \cdots
  \:\! C^{\,r_m}_{m,2m}
  \,\vec x
  \:\!=\:\!
		\left[\;\begin{matrix}
			x_1	\\[0.5ex] x_2	\\[0.5ex] \vdots	\\[0.5ex] x_m
		\\[0.5ex]
		\hline
		\\[-2ex]
		  x_{m+1} + a_{m+1}^{-1} a_1 x_1	\\[1ex] x_{m+2} + a_{m+2}^{-1} a_2 x_2	\\[1ex] \vdots \\[1ex]	x_{2m} + a_{2m}^{-1} a_m x_m
	  \end{matrix}\;\right]
  \:\!=\:\!
		\left[\;\begin{matrix}
			x_1	\\[0.5ex] x_2	\\[0.5ex] \vdots	\\[0.5ex] x_m
		\\[0.5ex]
		\hline
		\\[-2ex]
		  a_{m+1}^{-1} \gamma_1 \\[1ex] a_{m+2}^{-1} \gamma_2 \\[1ex] \vdots \\[1ex]	a_{2m}^{-1} \gamma_m
	  \end{matrix}\;\right]
  \:\!=:\:\!
		\left[\;\begin{matrix}
			x_1	\\[0.5ex] x_2	\\[0.5ex] \vdots	\\[0.5ex] x_m
		\\[0.5ex]
		\hline
		\\[-2ex]
		  \tilde \gamma_1 \\[1ex] \tilde \gamma_2 \\[1ex] \vdots \\[1ex]	\tilde \gamma_m
	  \end{matrix}\;\right].
\end{align}
Each coefficient $\tilde \gamma_j$ generates the same additive subgroup as $\gamma_j$ modulo $n$; if $d_1, \ldots, d_m$ are coefficients such that $x_j = d_m a_{j+m}^{-1} \gamma_j = d_m \tilde \gamma_j$, we then have
\begin{align}
	C^{-d_1}_{m+1,1}
%   \:\! C^{\,r_2}_{2,m+2}
  \:\! \cdots
  \:\! C^{-d_m}_{2m,m}
  \big[\; x_1 \; \cdots \; x_m \;\; \tilde\gamma_1 \; \cdots \; \tilde \gamma_m \;\big]\trans
  =\;
  \big[\; 0 \; \cdots \; 0 \;\; \tilde\gamma_1 \; \cdots \; \tilde \gamma_m \;\big]\trans.
\end{align}
The above performs the reduction of \eqref{eqn:localGCDcompute}, up to multiplicative units, in $O(m)$ symplectic row operations.
We may similarly emulate the reductions of \eqref{eqn:staircaseGCDcompute} and \eqref{eqn:symplecticReduceFirstColumn} in $O(m)$ symplectic row operations, using Lemma~\ref{lemma:Bezout+n} to reduce the computation of greatest common divisors (up to multiplicative unit factors) to performing powers of the operators $C_{j,k}$.

To summarize, using the bound on the diameter of the quadratic unitary graph $G_n$, we may refine the decomposition of symplectic operators in~\cite{HDM05} by substituting an explicit simulation of Euclid's algorithm by a constant-size sequence of symplectic operations.
This substitution provides an upper bound of  $O(m^2)$ for a decomposition of an operator $S \in \Sp_{2m}(\Z_n)$, a bound independent of the modulus $n$.

\egroup
\section{Remarks and open problems}

It should be noted that quadratic unitary graphs, while easy to describe, are closely tied to unsolved problems in computational complexity theory.
In particular, testing adjacency in a graph $G_n$ is precisely the quadratic residuacity problem, which has no known efficient algorithms and is considered unlikely to be efficiently solvable (see \eg\ Chapter~3 of \cite{MVvO96}).\footnote{%
  It should be noted that because quadratic residuacity can be reduced to integer factoring (by exploiting the Chinese Remainder theorem), and because factoring is solvable in a polynomial number of operations with bounded error with a \emph{quantum} computer~\cite{Shor}, testing adjacency in $G_n$ is also tractible for a quantum computer.}
Because of this, an efficient algorithm (deterministic or randomized) for discovering the shortest path between two vertices in $G_n$ should be considered unlikely.
We may then ask whether there are efficient algorithms for discovering ``short'' paths (having length bounded by a fixed constant) between vertices in $G_n$.

In Section~\ref{sec:perfect}, we provided a partially non-constructive proof that odd holes arise in quadratic unitary graphs $G_n$ which are odd but not a power of a prime $p \equiv 3 \pmod{4}$.
Numerical investigation suggests that, in particular, \emph{five-holes} (odd holes of size five) are very common in those $G_n$ which are not perfect graphs, even when restricting to five-holes involving the arc $0 \arc 1$.
It would be interesting to obtain a classification of all five-holes which occur in the imperfect graphs $G_n$\,.

As we noted in the introduction and in Section~\ref{sec:perfect}, the graphs $G_n$ for $n \equiv 1 \pmod{4}$ prime are also Paley graphs.
Shparlinski~\cite{Shpar06} shows that prime-order Paley graphs these graphs have high \emph{energy} (\ie\ the operator $1$-norm of the adjacency matrix), coming to within a factor of $(1 - \frac{1}{n})$ of the upper bound $\cE_{\max}(n) \,=\, \tfrac{1}{2}n(\sqrt n + 1)$ shown in~\cite{KM01} for graphs on $n$ vertices.
We may ask to what extent this and other properties of circulant Paley graphs generalize for quadratic unitary graphs.

The author would like to thank Charles Matthews, Robin Chapman, and Chris Godsil for helpful discussions.
This work was partly written during a visit to School of Computer Science at Reykjavik University.
This work was performed with financial support by MINOS EURONET and the EURYI scheme.

%  ======================================================================
\bibliographystyle{plainnat}

\appendix

\section{The existence of special B\'ezout coefficients}

For a sequence of integers $x_1, x_2, \ldots, x_k$, \emph{B\'ezout coefficients} are a corresponding sequence of integer coefficients $a_1, a_2, \ldots, a_k$ such that $\gcd(x_1, \ldots, x_k) \,=\, \sum a_j x_j$; the existence of such a sequence of coefficients $a_1, \ldots, a_k$ is implied by the ``simple'' Euclidean algorithm.

Consider the greatest common divisior of a sequence $x_1, \ldots, x_k$ together with another integer $n$: this is equivalent to computing $\gamma = \gcd(x_1, \ldots, x_k)$ modulo $n$ via Euclid's algorithm.
We may compute greatest common divisors modulo $n$ recursively, by computing $\gamma_2 = \gcd(x_1, x_2)$ modulo $n$, then $\gamma_3 = \gcd(\gcd(x_1, x_2), x_3)$ modulo $n$, and so forth.
However, for each intermediate stage $1 < j < k$, it is not necessary to obtain $\gamma_j$ itself, but instead a similar residue $\tilde \gamma_j$ which generates the same subgroup modulo $n$; by definition, the set of integer combinations modulo $n$ of such an integer $\tilde\gamma_j$ is the same as the set of integer combinations of $\gamma_j$, so that $\gcd(a,\tilde\gamma_j) \equiv \gcd(a,\gamma_j) \pmod{n}$ for any $a \in \Z$.

The simplest application of this observation is that in $\Z_n$, any integer $x$ may serve as a substitute for its own greatest common divisor with $n$:
\begin{lemma}
	\label{lemma:Bezout+2}
 	Let $\gamma = \gcd(x,n)$ for $x, n \in \Z$: then there exist $a,b \in \Z$ such that $ax + bn = \gamma$ and where $a$ is relatively prime to $n$.
\end{lemma}
\begin{proof}
	Consider arbitrary $a, b \in \Z$ such that $ax + bn = \gamma$.
	Let $\alpha = \gcd(a,\gamma)$: then $\alpha$ divides $a$, $x$, and $n$.
	We have $\frac{a}{\alpha} x = (\frac{\gamma}{\alpha} - b \frac{D}{\alpha}) \in \frac{\gamma}{\alpha} \Z$.
 	By construction, $\frac{a}{\alpha}$ is an integer relatively prime to $\frac{\gamma}{\alpha}$\,: thus, $x$ is a multiple of $\frac{\gamma}{\alpha}$.
	Let $m = \frac{\alpha x}{\gamma}$\,: then $\frac{m}{\alpha} = \frac{x}{\gamma} \in \Z$, and furthermore is relatively prime to $n$.
	Then, if we let $\bar a \in \Z$ be such that $\frac{\bar a m}{\alpha} \equiv 1 \pmod{n}$, we have $\bar a x \equiv \gamma \pmod{n}$ as required.
\end{proof}

We generalize the above lemma as follows.
In order to compute a suitable integer $\tilde\gamma_j$ which generates the same additive group (modulo $n$) as $\gamma_j = \gcd(\tilde\gamma_{j-1}, x_j)$ for each $j$, we may compute B\'ezout coefficients $a, b, c$ such that
\begin{gather}
	ax_j + b\tilde\gamma_{j-1} + cn = \gamma_j.
\end{gather}
If we may find such a set of coefficients that $a$ is coprime to $n$, we then have
\begin{gather}
 	x_j \,+\, \tilde a b \tilde\gamma_{j-1} \;\equiv\; \tilde a\gamma_j \pmod{n}	\;,
\end{gather}
where $a \tilde a \equiv 1 \pmod{n}$\,, in which case we may let $\tilde\gamma_j = \tilde a \gamma_j$.
That is, if such $a \in \Z_n\units$ exists, we may compute $\tilde\gamma_j$ as the sum of $x_j$ with some multiple of $\tilde\gamma_{j-1}$, which can be computed using a single addition operation and a single scalar multiplication.
We show that such B\'ezout coefficients may always be found by proving Lemma~\ref{lemma:Bezout+n} (page~\pageref{lemma:Bezout+n}):

\begin{lemma}
	\label{lemma:Bezout+n'}
 	Let $\gamma = \gcd(x, y, n)$: then there exist $a,b,c \in \Z$ such that $ax + by + cn = \gamma$ and where both $a$ and $b$ are relatively prime to $n$.
\end{lemma}

\begin{proof}
	Let $x' = \gcd(x, n)$ and $y' = \gcd(y, n)$: by Lemma~\ref{lemma:Bezout+2}, we then have $x \equiv u_x x' \pmod{n}$ and $y \equiv u_y y' \pmod{n}$ for multiplicative units $u_x, u_y \in \Z_{n}\units$, and $\gamma = \gcd(x', y')$.
	Define
	\begin{align}
		\bar x \;=&\;\, \frac{x'}{\gamma}\,,
	&
		\bar y \;=&\;\, \frac{y'}{\gamma}\,;
	\end{align}
	these are both divisors of $n$, and form a relatively prime pair.
	We may then partition the prime factors of $n$ into those which divide $\bar x$, those which divide $\bar y$, and those which divide neither.
	Let $N_x$ be the largest factor of $\bar n$ whose prime factors divide $\bar x$, $N_y$ be the largest factor of $\bar n$ whose prime factors divide $\bar y$, and $N_n = n / N_x N_y$\,: then $N_x$ and $N_y$ are coprime, so that $N_n$ is also an integer and relatively prime both to $N_x$ and $N_y$.
	We then have $n = N_x N_y N_n$.

	As $\bar x$ and $\bar y$ are coprime, there exist integers $a,b \in \Z$ such that $a\bar x + b \bar y = 1$.
	Note that $a$ is coprime to $\bar y$\,, from which it follows that $a$ is coprime to $N_y$ as well, as $N_y$ and $\bar y$ have the same prime factors; similarly, $b$ is coprime to $N_x$\,.
	Let
	\begin{align}
			h
		\;=&\;
			\begin{cases}
				0		\;,	&	\text{if $\gcd(a, N_x N_d) = \gcd(b, N_y N_n) = 1$};		\\[0.5ex]
				N_y	\;,	&	\text{if $\gcd(a, N_x N_d) > 1$, but $\gcd(b, N_y N_n) = 1$};	\\[0.5ex]
				N_x	\;,	&	\text{if $\gcd(a, N_x N_d) = 1$, but $\gcd(b, N_y N_n) > 1$};	\\[0.5ex]
				1		\;,	&	\text{otherwise};
			\end{cases}
	\end{align}
	and let $\alpha = a + h \bar y$ and $\beta = b - h \bar x$.
	If $a$ has prime factors in common with $N_x N_n$, then $\alpha$ does not, by the fact that both $\bar y$ and $\bar y N_y$ are relatively prime to $N_x N_n$\,; otherwise, $\alpha$ is coprime to $N_x N_n$ anyway by the coprimality of $a$ to $N_x N_n$.
	In either case, we also have $\alpha$ coprime to $N_y$, by the coprimality of $a$ and $\bar y$.
	Thus, $\alpha$ is relatively prime to $n = N_x N_y N_n$; and similarly, $\beta$ is coprime to $n$.
	We may then observe that
	\begin{gather}
	 		\alpha \bar x	+	\beta \bar y
		\;=\;
			(a + h \bar y)\bar x  + (b - h \bar x) \bar y
		\;=\;
			a \bar x + b \bar y
		\;=\;
			1\;,
	\end{gather}
	from which it follows that $\alpha x' + \beta y' = \gamma$.
	Let $\bar a , \bar b \in \Z$ be such that $u_x \bar a \equiv \alpha \pmod{n}$ and $u_y \bar b \equiv \beta \pmod{n}$\,: then, we have
	\begin{gather}
	 		\bar a x + \bar b y
		\;\equiv\;
			\alpha x' + \beta y'
		\;=\;
			\gamma
		\pmod{n}:
	\end{gather}
	as $\alpha$, $u_x$, $\beta$, and $u_y$ are all coprime to $n$, both $\bar a$ and $\bar b$ are also coprime to $n$.
\end{proof}

\end{document}

%% file: preamble.tex
\usepackage[numbers,sort&compress]{natbib}
\usepackage{aliascnt}
\usepackage{jrnbsymb}
\usepackage{amsthm}
\usepackage{amsmath}
\usepackage{enumitem}

\theoremstyle{definition}
\newtheorem{definition}{Definition}

\theoremstyle{plain}   
\newtheorem{theorem}{Theorem}

\newaliascnt{lemma}{theorem}
\newaliascnt{proposition}{theorem}
\newtheorem{lemma}[lemma]{Lemma}

\aliascntresetthe{lemma}
\aliascntresetthe{proposition}

\newtheorem{corollary}{Corollary}[theorem]

\theoremstyle{definition}
\newtheorem*{convention}{Convention}
\newtheorem*{notation}{Notation}

\newtheoremstyle{namedtheorem}% name
  {3pt}%      Space above, empty = `usual value'
  {3pt}%      Space below
  {\itshape}% Body font
  {}%         Indent amount (empty = no indent, \parindent = para indent)
  {\bfseries}% Thm head font
  {.}%        Punctuation after thm head
  {.5em}%     Space after thm head: " " = normal interword space;
  {\thmnote{#3}}% Thm head spec

\theoremstyle{namedtheorem}
\newtheorem*{named theorem}{Named Theorem}

\newlength\romanumlabelwd
\newenvironment{romanum}{%
	\begin{enumerate}[label=\textup{(\textit{\roman*})},labelwidth=\romanumlabelwd,leftmargin=1\romanumlabelwd,itemsep=0.2ex]%
}{%
	\end{enumerate}%
}

\renewenvironment{cases}[1][c]{%
 	\left\{ \begin{array}{#1@{\quad}l}%
}{%
	\end{array} \right.%
}

\def\texttilde{\raise.2ex\hbox{\m@th$\scriptstyle\sim$}}
\def\@nbsp{~}
\newcommand\arXiv[1][]{[{\small\texttt{\bfseries arXiv:#1}}]}

\NewNameCommandStyle{italabbr}{\italabbr{#1}}

\makeatletter
\newcommand\italabbr[1]{{\itshape \@italabbr#1\relax}}
\def\@italabbr#1{%
	\if#1\relax\def\@tempa{\/}\else\def\@tempa{#1.\@italabbr}\fi
	\@tempa}
\makeatother

\NameMathOperators{\diam, \udiam, \Sp, \GL, \Cay}
\namecommands[italabbr]{\ie, \eg}
\namecommand\cE

\newcommand\units{^{\x}}
\newcommand\arc{\rightarrow}
\newcommand\cra{\leftarrow}

\renewcommand\vec[1]{\mathbf{#1}}
\newcommand\unit[1][e]{\hat{\vec #1}}

\newcommand\paren[1]{\left( #1 \right)}
\newcommand\sqparen[1]{\left[ #1 \right]}
\newcommand\ens[1]{\left\{ #1 \right\}}

\let\gen\anglebr

\newcommand\supp[1]{^{\paren{#1}}}

\newcommand\card[1]		{\left\lvert #1 \right\rvert}							% Cardinality

\newcommand\abs\card

%% file: sp2decomposn.bbl
\begin{thebibliography}{alpha}

\bibitem{EE89}	P.~Erd{\H o}s, A.~B.~Evans.
\newblock		\emph{Representations of graphs and orthogonal Latin square graphs}.
\newblock		J. Graph Theory \textbf{13} (pp.~593--595), 1989.

\bibitem{DG95}	I.~Dejter, R.~E.~Giudici.
\newblock		\emph{On unitary Cayley graphs}.
\newblock		J. Combin. Math. Combin. Comput. \textbf{18} (pp.~121--124), 1995.

\bibitem{BG04}	P.~Berrizbeitia, R.~E.~Giudici.
\newblock		\emph{On cycles in the sequence of unitary Cayley graphs}.
\newblock		Discrete Math. \textbf{282} (pp.~1--3), 2004.

\bibitem{KS07}		W.~Klotz, T.~Sander.
\newblock			\emph{Some Properties of Unitary Cayley Graphs}.
\newblock			Elec. J. Combinatorics \textbf{14}, 2007.

\bibitem{RV09}		H.~N.~Ramaswamy, C.~R.~Veena.
\newblock			\emph{On the Energy of Unitary Cayley Graphs}.
\newblock			Elec. J. Combinatorics \textbf{16}, 2009.

\bibitem{W62}		P.~M.~Weichsel.
\newblock			\emph{The Kronecker Product of Graphs}.
\newblock			Proc. of the AMS \textbf{13} (pp. 47--52), 1962.

\bibitem{Gauss}		C.~F.~Gauss.
\newblock			\emph{Disquisitiones Arithmetic\ae --- English Edition}.
\newblock			Springer-Verlag, New York-Heidelberg, 1986.

\bibitem{KS09}		S.~Klavzar, S.~Severini.
\newblock			\emph{Tensor 2-sums and entanglement}.
\newblock			Preprint \arXiv[0909.1039], 2009.

\bibitem{Aladov1896}	N.~S.~Aladov.
\newblock				\emph{On the distribution of quadratic residues and nonresidues of a prime number
   $p$ in the sequence $1, 2, \ldots, p - 1$}.
\newblock	Mat. Sbornik \textbf{18} (pp. 61--75), 1896.
\newblock	(Russian)

\bibitem{AR01}		J.~L.~R.~Alfonsin, B.~A.~Reed.
\newblock			\emph{Perfect Graphs}.
\newblock			John Wiley \& Sons, New York-Chichester-Brisbane, 2001.

\bibitem{CRST02}	M.~Chudnovsky, N.~Robertson, P.~Seymour, R.~Thomas.
\newblock			\emph{The strong perfect graph theorem}.
\newblock			Annals of Math. \textbf{164}, (pp.~51--229), 2002.

\bibitem{Berge}		C.~Berge.
\newblock			\emph{F{\"a}rbung von Graphen, deren s{\"a}mtliche bzw. deren ungerade Kreise starr sind}.
\newblock			Wiss. Z. Martin-Luther-Univ. Halle-Wittenberg Math.-Natur. Reihe \textbf{10} (pp.~114--115), 1961.

\bibitem{MP06}		E.~Maistrelli, D.~B.~Penman.
\newblock			\emph{Some colouring problems for Paley graphs}.
\newblock			J. Discrete Math. \textbf{306} (pp.~99--106), 2006.

\bibitem{B1925}		A.~A.~Bennett.
\newblock			\emph{On sets of three consecutive integers which are quadratic residues of primes}.
\newblock			Bull. Amer. Math. Soc. \textbf{31} (pp.~411--412), 1925.

\bibitem{HDM05}		E.~Hostens, J.~Dehaene, and B.~De~Moor.
\newblock			\emph{Stabilizer states and Clifford operations for systems of arbitrary dimensions and modular arithmetic}.
\newblock			Phys. Rev. A \textbf{71} (042315), 2005.
\newblock			\arXiv[quant-ph/0408190]

\bibitem{MVvO96}	A.~J.~Menezes, S.~A.~Vanstone, P.~C.~Van~Oorschot.
\newblock			\emph{Handbook of Applied Cryptography}.
\newblock			CRC Press, Inc., Boca Raton, 1996.

\bibitem{Shor}		P.~W.~Shor.
\newblock			\emph{Polynomial-Time Algorithms for Prime Factorization and Discrete Logarithms on a Quantum Computer}.
\newblock			SIAM J. Sci. Statist. Comput. \textbf{26} (pp. 1484--1509), 1997.

\bibitem{Shpar06}	I.~Shparlinski.
\newblock			\emph{On the energy of some circulant graphs}.
\newblock			Linear Algebra Appl. \textbf{414} (pp. 371--382), 2006.

\bibitem{KM01}		J.~H.~Koolen, V.~Moulton.
\newblock			\emph{Maximal energy graphs}.
\newblock			Adv. Appl. Math. \textbf{26} (pp. 47--52), 2001.

\end{thebibliography}
